\documentclass{article}
\usepackage[a4paper, portrait, margin=1.1811in]{geometry}
\usepackage[english]{babel}
\usepackage[utf8]{inputenc}
\usepackage[T1]{fontenc}
\usepackage{helvet}
\usepackage{etoolbox}
\usepackage{graphicx}
\usepackage{titlesec}
\usepackage{caption}
\usepackage{booktabs}
\usepackage{xcolor} 
\usepackage[colorlinks, citecolor=cyan]{hyperref}
\usepackage{caption}
\graphicspath{ {./images/} }
\usepackage{scrextend}
\usepackage{fancyhdr}
\usepackage{graphicx}
\usepackage{amssymb}
\usepackage{amsmath}
\usepackage{amsthm}
\usepackage{stmaryrd}
\newtheorem{theorem}{Theorem}[section]
\newtheorem{definition}[theorem]{Definition}
\newtheorem{corollary}[theorem]{Corollary}
\newtheorem{lemma}[theorem]{Lemma}

\newtheorem{example}[theorem]{Example}
\newtheorem{proposition}[theorem]{Proposition}

\fancypagestyle{plain}{
	\fancyhf{}

}

\makeatletter
\patchcmd{\@maketitle}{\LARGE \@title}{\fontsize{16}{19.2}\selectfont\@title}{}{}
\makeatother

\usepackage{authblk}

\newsavebox\affbox
\author[1*]{\textbf{Xiaotong Liu}}
\author[2]{\textbf{Yiyu Liang}}
\affil[1,2]{ School of Mathematics and Statistics,
Beijing Jiaotong University, Beijing 100044,
PEOPLE’S REPUBLIC OF CHINA
}

\titlespacing\section{0pt}{12pt plus 4pt minus 2pt}{0pt plus 2pt minus 2pt}
\titlespacing\subsection{12pt}{12pt plus 4pt minus 2pt}{0pt plus 2pt minus 2pt}
\titlespacing\subsubsection{12pt}{12pt plus 4pt minus 2pt}{0pt plus 2pt minus 2pt}

\titleformat{\section}{\normalfont\fontsize{10}{15}\bfseries}{\thesection.}{1em}{}
\titleformat{\subsection}{\normalfont\fontsize{10}{15}\bfseries}{\thesubsection.}{1em}{}
\titleformat{\subsubsection}{\normalfont\fontsize{10}{15}\bfseries}{\thesubsubsection.}{1em}{}

\titleformat{\author}{\normalfont\fontsize{10}{15}\bfseries}{\thesection}{1em}{}

\title{\textbf{\huge A signal recovery guarantee with Restricted Isometry Property and Null Space Property for weighted $l_1$ minimization}}
\date{2023/09}    

\begin{document}

\pagestyle{headings}	
\newpage
\setcounter{page}{1}
\renewcommand{\thepage}{\arabic{page}}

\captionsetup[figure]{labelfont={bf},labelformat={default},labelsep=period,name={Figure }}	\captionsetup[table]{labelfont={bf},labelformat={default},labelsep=period,name={Table }}
\setlength{\parskip}{0.5em}
	
\maketitle
	
\noindent\rule{15cm}{0.5pt}
	\begin{abstract}
		\textbf{Abstract } Signal reconstruction is a crucial aspect of compressive sensing. In weighted cases, there are two common types of weights.  In order to establish a unified framework for handling various types of weights, the sparse function is introduced. By employing this sparse function, a generalized form of the weighted null space property is developed, which is sufficient and necessary to exact recovery through weighted $\ell_1$ minimization. This paper will provide a new recovery guarantee called $\omega$-RIP-NSP with the weighted $\ell_1$ minimization, combining the weighted null space property and the weighted restricted isometry property. The new recovery guarantee only depends on the kernel of matrices and provides robust and stable error bounds. The third aim is to explain the relationships between $\omega$-RIP, $\omega$-RIP-NSP and $\omega$-NSP. $\omega$-RIP is obviously stronger than $\omega$-RIP-NSP by definition. We show that $\omega$-RIP-NSP is stronger than the weighted null space property by constructing a matrix that satisfies the weighted null space property but not $\omega$-RIP-NSP.

		\let\thefootnote\relax\footnotetext{
			\small $^{*}$\textbf{Corresponding author.} \textit{
				\textit{E-mail address: \color{cyan}xiaotongl@bjtu.edu.cn}}\\
			    Yiyu Liang is supported by the National Natural Science Foundation of China (Grant Nos. 12271021 and 11971402).
			}
		\textbf{\textit{Keywords}}: \textit{Compressive sensing; Weighted $\ell_1$ minimization; Restricted isometry property; Null space property; Weighted sparse}

        \textbf{MSC}: 15A12 94A12 47A52
	\end{abstract}
\noindent\rule{15cm}{0.4pt}

\section{Introduction}

In recent years, compressive sensing\cite{donoho2006}\cite{candesorigin}\cite{candesorigin2} has emerged as a rapidly developing field in both theory and applications. A comprehensive introduction to compressive sensing can be found in \cite{j2014A}. The main motivation of compressive sensing is to efficiently reconstruct high-dimensional signals from a small number of measurements. This is achieved by solving an underdetermined linear system, where the signals are assumed to be sparse or compressible in some domain. The process can be described mathematically as solving $y=Ax$, where $y$ is the vector of measurements, $x$ is the unknown signal, and $A\in\mathbb{C}^{m\times N}$ is a sensing matrix with $m<N$. The success of reconstruction is ensured by the well-defined properties of the sensing matrix, such as the Restricted Isometry Property (RIP) and the Null Space Property (NSP). The RIP characterizes $\|Ax\|\approx\|x\|$, while the NSP characterizes the properties of all sparse or compressible signals in the null space. These properties provide a powerful framework for ensuring their robustness and stability in reconstructing the signals.

For some structured measurement matrices, with prior knowledge of the matrices or signals, the weights can be involved to emphasize or de-emphasize certain parts of the signal, and it is a valuable tool for improving the performance of the reconstruction and has been studied in various contexts\cite{kha2009}\cite{vas2009}\cite{von2007}.

Over the past decade, the weighted $\ell_1$ minimization program has drawn large numbers of researchers to pay their attention to sparse signal recovery \cite{rauhut2016}\cite{CHEN2019417}\cite{CHEN201883}\cite{fried2012}. It can be formulated as 
\begin{equation}
\label{omegal1}
    \min_{x\in\mathbb{C}^N}\|x\|_{\omega,1}\,\text{subject to}\,\|y-Ax\|_{2}\leq\epsilon.
\end{equation}
Where $x\in\mathbb{C}^N$ is the aim vector, $y\in\mathbb{C}^m$ is the measurements vector and $\epsilon>0$ is the noise bound and $\|x\|_{\omega,1}=\sum_{i=1}^N\omega_i|x_i|$.

There are various types of weights that can be used in compressive sensing. In some applications, it may be possible to estimate the support of the signal or its largest coefficients, so there are weights with the prior information support. For example, in video signals, temporal correlation can be exploited to estimate the support using previously decoded frames\cite{fried2012}. In this case, the weight vector is $\mathbb{\omega}\in[0,1]^n$ and the sparsity is the same as the standard sparsity, which is defined as the number of nonzero entries on the set\cite{ge2021}\cite{ge2021weighted}\cite{zhou2013}.  For that type of weights, \cite{fried2012}\cite{peng2014}\cite{yu2013}\cite{ge2021}\cite{ge2021weighted} have explored sufficient conditions for weighted $\ell_1$ minimization using standard RIP and NSP and offered improved error bounds or RIP constants bounds. However, the use of standard RIP and NSP limits the extent to which recovery guarantees can improve on those for unweighted $\ell_1$ minimization. To further improve error estimation, the investigation of the weighted $\ell_1$ minimization necessitates some extended versions of RIP and NSP. Zhou et al. \cite{zhou2013} proved that one of the extended versions of NSP, which is called the weighted null space property, is the sufficient and necessary condition for exact recovery, and provided an algorithm for the weighted $\ell_1$ minimization.

Motivated by application to function interpolation, another type of weights has been investigated in\cite{rauhut2016}\cite{huo2018}, where the weights are considered into the sparsity structure.  And the weights are described as $\omega_i\geq1$ and the sparsity is the sum of the square of the weights on the set. In \cite{rauhut2016}, Rauhut introduced the extended properties of this type of weights, including  the weighted null space property, the weighted robust null space property and the weighted restricted isometry property.

While both the weighted restricted isometry property and the weighted null space property provide a recovery guarantee for signal reconstruction, they differ in some aspects. The weighted null space property is a necessary and sufficient condition for perfect recovery in noiseless cases and only depends on the null space of the measurement matrix. the weighted restricted isometry property provides a criterion for measuring the performance of the measurement matrix because of the easier theoretical justification than the weighted null space property. The weighted restricted isometry property also can provide a sufficient condition for stable and robust error bound of weighted $\ell_1$ minimization.

It is noteworthy that if the matrix $A$ satisfies the weighted null space property, the signal $x$ can be exactly reconstructed via the undetermined system $Ax=y$. Importantly, the applicability of the weighted null space property is rooted in the null space itself. Therefore, even if the system undergoes stretching, as long as the null space remains unchanged, the recovery for $x$ remains attainable.
Specifically,  if signal $x$ can be exactly reconstructed within the system $Ax=y$, it should also remain recoverable within the system $2Ax=2y$. In this context, both matrices $A$ and $2A$ share the same null space. However, even though $A$ and $2A$ have identical null spaces, $2A$ may not satisfy the weighted restricted isometry property when subjected to the fixed restricted isometry constant associated with $A$.
Contrary to the weighted null space property which only depends on the null space, the properties of the weighted restricted isometry property is susceptible to distortion when the matrices are stretched. Therefore, they cannot solely depend on the null space, nor accurately reflect the invariant property of the system.

From noisy measurements, the robustness of  the reconstruction of compressible signals using weighted $\ell_1$ minimization can not be guaranteed by the weighted null space property\cite{zhou2013}. However, if matrix $A$ satisfies the weighted restricted isometry property, even with the noisy measurements and compressible but not sparse signal, the reconstruction can still be assured with a robust and stable error bound.

 To unitize different types of weights, we generalize the properties, which include the generalized weighted null space ($\omega$-NSP, see in Definition \ref{ensp}) and the generalized weighted restricted isometry property ($\omega$-RIP, see in Definition \ref{wrip}),  using a \textit{ sparse function} (as discussed in Section \ref{2}). 
In section \ref{section2}, the first result for this paper proves that $\omega$-NSP is the sufficient and necessary condition of the exact recovery via the weighted $\ell_1$ minimization.

In Section \ref{section3}, the main point is to find a new property combining the advantages of $\omega$-NSP and $\omega$-RIP, which means it only depends on the null space of matrices, and it guarantees stable and robust recovery.
Inspired by \cite{cahill2016} and the characteristic of $\omega$-RIP and $\omega$-NSP, the combination of  $\omega$-RIP and $\omega$-NSP, denoted by $\omega$-RIP-NSP. 
If a matrix $A$ satisfies $\omega$-RIP, it follows that any matrix sharing the same null space as $A$ satisfies the $\omega$-RIP-NSP (see in Definition \ref{wripnsp}).
According to the definition, it is easy to know that $\omega$-RIP-NSP only depends on the null space, and the latter needs to be proven, which is one of the main results of this work:
If $\omega$-RIP provides stable and robust error bounds for the weighted $\ell_1$ minimization,  it can be inferred that the $\omega$-RIP-NSP property will also ensure stable and robust reconstruction through the weighted $\ell_1$ minimization (see in Theorem \ref{thmbd}).
In two special cases mentioned above, we have new results implied by Theorem \ref{thmbd}.
In the case of weights where sparsity is defined by the sum of the squares of the weights, Rauhut provides results indicating that $\omega$-RIP provides stable and robust error bounds\cite{rauhut2016}. So the new bounds guaranteed by $\omega$-RIP-NSP are gained.
For the weights with the standard sparsity, there are the results that $\omega$-RIP can guarantee stable and robust reconstruction with the weights have the same value of constants on all components\cite{zhou2013}. In  Subsection \ref{ex2}, we present that $\omega$-RIP can guarantee stable and robust reconstruction even when the weights exhibit varying values, all of which are less than 1, across their components (see in Theorem \ref{01th}).
Moreover, the constants featured in these error bounds are linked to the specific type of weights being considered. This connection arises from the relationship between the number of partitions, as mentioned in the Lemma \ref{op}, and the method used to define sparsity for that particular type of weights.

The third objective is to explain the interrelationship between $\omega$-RIP, $\omega$-RIP-NSP and $\omega$-NSP. In Section \ref{section4}, the first result is that $\omega$-RIP-NSP strictly contains $\omega$-RIP (see in Proposition \ref{rrn}).
The proof of $\omega$-NSP strictly includes $\omega$-RIP-NSP is achieved by proving there exist matrices that satisfy $\omega$-NSP but fail to satisfy $\omega$-RIP-NSP (see in Theorem \ref{1.3} and Theorem \ref{result2}).
The proof proceeds in a similar manner to those in \cite{cahill2016},  with certain modifications. 
In order to construct the counter-example, it is essential to narrow the focus to the specific type of weights under consideration.
For the weights with the sparsity defined by the sum of the square of the weights on the set, instead of utilizing partial discrete Fourier matrices, the proof utilizes general unitary matrices to construct the requisite matrices.  
For the weights with standard sparsity, according to the result presented in Theorem \ref{01th}, we used the partial Fourier matrix to construct the matrix satisfying $\omega$-NSP with the arbitrarily chosen weights.
Compared with the proof in \cite{cahill2016}, the current proof is more extensive, as it accommodates arbitrary choices of weights, and covers a broader category of bounded orthonormal systems.

For the weights with standard sparsity, another robust property called $\omega$-robust-NSP, is involved. We force to show that $\omega$-RIP-NSP and $\omega$-robust-NSP are not included by each other. Assume matrix $\Phi$ satisfies $\omega$-robust-NSP, then for any invertible matrix $U$, $U\Phi$ shares the same null space with $\Phi$. Then for some matrices $U$, $U\Phi$ will violate $\omega$-robust-NSP (see in Theorem \ref{r2n}). 
On the other hand, identifying the set of matrices that satisfy $\omega$-robust-NSP but not $\omega$-RIP can be challenging. This remains an open question, even in the unweighted case, where constructing matrices that fulfill robust null space property but not RIP is still an area of active research.

Let us introduce some notations that will be mentioned in the following paper. The usual $\ell_\infty$ norm in $\mathbb{C}^N$ is denoted by $\|x\|_\infty=\max_{j=1,\ldots,N}|x_j|$. And $\|x\|_0:=\#\{j:x_j\ne 0\}$. The null space of matrix $A$ is denoted by $\ker(A)$. The index set of support of vector $x\in\mathbb{C}^N$ is denoted by $[N]={1,2,\ldots,N}$. The complement set of set $S$ is $S^c$. And $x_S$ is the vector has the same entries as $x$ on $S$, and 0 on $S^c$.

\section{The generalized weighted null space property}
\label{section2}
In this section, we introduce the generalized weighted null space property which guarantees exact recovery through weighted $\ell_1$ minimization. 
The weights typically take the form of positive constants assigned to the elements of a vector's index set. These weights define a corresponding weighted norm, which quantifies the importance or relevance of each element in the vector.
For $x\in\mathbb{C}^N$ and weight $\omega=[\omega_1,\ldots,\omega_N]^T,\;\omega_j\geq0,\;j=1,\ldots,N$, the definition of the weighted $\ell_1$ norm is $\|\boldsymbol{x}\|_{\omega,1}=\sum_{j=1}^{N}\omega_j|x_j|$.

The differences between types of weights also include the definition of sparsity. To achieve unification, we introduce the sparse function. For any weight $\omega=[\omega_1,\ldots,\omega_N]^T,\;\omega_j\geq0,\;j=1,\ldots,N$, and the index set $S\subset[N]$, \textit{sparse function}  $\nu_{\omega}(S)$ is the function has following properties:
\begin{itemize}
    \item $\nu_{\omega}(S)\geq0$;
    \item For the discrete sets $S_1\cap S_2=\emptyset$, $\nu_\omega(S_1\cup S_2)=\nu_\omega(S_1)+\nu_\omega(S_2)$.
\end{itemize}
The most useful examples are that $\nu_{\omega}(S)=|S|$ or $\nu_{\omega}(S)=\sum_{i\in S}\omega_i^2$.
 We call a vector $x\in\mathbb{C}^N$ \textit{weighted s-sparse} with positive integer $s$, if for the support $S=supp(x)$ of $x$, $\nu_{\omega}(S)\leq s$. 
Conditions of the form $s>\max_{i\in[N]}\nu_{\omega}(\{i\})$ are somewhat natural in the context of weighted sparse approximations, as those indices with weights $s<\max_{i\in[N]}\nu_{\omega}(\{i\})$ cannot possibly contribute to the subset satisfying weighted $s$-sparse. For the rest of the paper, we assume that $s\geq2\max_{i\in[N]}\nu_{\omega}(\{i\})$ for technical reasons.

Moreover, since NSP involves solely the weighted $\ell_1$ norms of vectors in the kernel of measurement matrices, a version of the generalized weighted null space property, denoted by $\omega$-NSP, can be formulated to unify the weighted null space properties for different types of weights.

Additionally, let set S be a subset of $[N]$,  $\boldsymbol{x}_S$ denotes the restriction of $\boldsymbol{x}$ to S, which is defined to be the vector that equals $\boldsymbol{x}$ on S, and zero elsewhere. 
\begin{definition}[$\omega$-NSP]
\label{ensp}
    For given weights $\omega$,  the matrix $A\in \mathbb{C}^{m\times{N}}$ is said to satisfy the generalized weighted null space property ($\omega$-NSP) of order $s$ if it satisfies:

    \begin{equation*}
\|\upsilon_{S}\|_{\omega,1}<\|\upsilon_{S^{C}}\|_{\omega,1}\;\mbox{for all}\;\upsilon\in\ker(A)\;\mbox{and}\;S\subset[N]\;with\;\nu_{\omega}(S)\leq s.
\end{equation*}
where $\nu_{\omega}(S)$ is a sparse function of the set $S$.
\end{definition}

There is another equivalent form of $\omega$-NSP with the \textit{null space constant} $\gamma$, the equivalence can be proven with the compactness of $\ker(A)\cap\mathbb{S}^{d-1}$, where $\mathbb{S}^{d-1}$ is a unit sphere.
\begin{proposition}
    For given weights $\omega$ and constant $\gamma\in(0,1)$,  the matrix $A\in \mathbb{C}^{m\times{N}}$ is said to satisfy the generalized weighted null space property ($\omega$-NSP) of order $s$ with $\gamma$ if it satisfies:

    \begin{equation*}
\|\upsilon_{S}\|_{\omega,1}\leq\gamma\|\upsilon_{S^{C}}\|_{\omega,1}\;\mbox{for all}\;\upsilon\in\ker(A)\;\mbox{and}\;S\subset[N]\;with\;\nu_{\omega}(S)\leq s.
\end{equation*}
where $\nu_{\omega}(S)$ is a sparse function of the set $S$.
\end{proposition}

Similar to the standard null space property, $\omega$-NSP also is a sufficient and necessary condition for the exact recovery of noiseless measurements. The result is described in the following theorem. And what calls special attention is that $\nu_{\omega}(S)$ must be consistent with that in Definition \ref{ensp}.
\begin{theorem}
\label{ewl1}
    For given weights $\omega$, every vector $\hat{x}\in\mathbb{C}^N$ whose support $S$ satisfies $\nu_{\omega}(S)\leq s$ as in Definition \ref{ensp} is the unique solution of 
    \begin{equation}
    \label{exl1}
        \min_{z\in\mathbb{C}^N}\|z\|_{\omega,1}\,\text{subject to}\ \Phi\hat{x}=\Phi z,
    \end{equation}
    if and only if $\Phi$ satisfies the $\omega$-NSP of order $s$.
\end{theorem}

\begin{proof}
    \textbf{(Sufficiency)} Let us assume $\Phi$ satisfies $\omega$-NSP.\\
    Suppose $\hat{x}\in\mathbb{C}^N$ whose support $S$ satisfies $\nu_{\omega}(S)\leq s$ as in Definition \ref{ensp}  and $z\in\mathbb{C}^N$ such that $\Phi\hat{x}=\Phi z$. Then $v=\hat{x}-z\in\ker(\Phi)$ and $\hat{x}_S=\hat{x}$. These together with $\Phi$ satisfying $\omega$-NSP, further imply
    \begin{align*}
        \|\hat{x}\|_{\omega,1}&\leq\|\hat{x}-z_S\|_{\omega,1}+\|z_S\|_{\omega,1}\\
        &=\|(\hat{x}-z)_S\|_{\omega,1}+\|z_S\|_{\omega,1}\\
        &<\|v_{S^c}\|_{\omega,1}+\|z_S\|_{\omega,1}\\
        &=\|z_{S^c}\|_{\omega,1}+\|z_S\|_{\omega,1}=\|z\|_{\omega,1}.
    \end{align*}
    This established the required minimality of  $\|\hat{x}\|_{\omega,1}$.\\
    \textbf{(Necessity)}  For any vector $v\in\ker(\Phi)$ and any set $S$ satisfies $\nu_{\omega}(S)\leq s$ as in Definition \ref{ensp}, $\Phi v=0$ implies that $\Phi v_S=-\Phi v_{S^c}$. Regarding $v_S$ as $\hat{x}$ in \eqref{exl1}, by the minimality we have 
        $$
        \|v_S\|_{\omega,1}<\|v_{S^c}\|_{\omega,1}.    
    $$ 
\end{proof}

The condition $\nu_{\omega}(S)\leq s$ is an abstract constraint that does not specify the exact form of the weights. Previous literature has provided two examples of weights belonging to $[1,+\infty)$ and $[0,1)$, respectively. Based on these examples, a third example has been proposed that does not impose any restrictions on the values of the weights.

\begin{example}
 As special cases, for the case $\omega_i\geq1$ which has been mentioned above\cite{rauhut2016}\cite{huo2018}, sparsity is defined in terms of the weighted cardinality of set $S$, denoted by $\omega(S)=\sum_{j\in S}\omega_j^2\leq s$, such that it must be less than s. In this case, $\nu_{\omega}(S)$ can be defined as
$$
\nu_{\omega}(S)=\omega(S)=\sum_{j\in S}\omega_j^2.
$$

In this case, Theorem \ref{ewl1} can be described as follow. Every vector $\hat{x}\in\mathbb{C}^N$ satisfying $\sum_{\{j:\hat{x}_j\neq0\}}\omega_j^2\leq s$ is the unique solution of 
    \begin{equation}
    \label{exl1}
        \min_{x\in\mathbb{C}^N}\|x\|_{\omega,1}\,\text{subject to}\ \Phi\hat{x}=\Phi x,
    \end{equation}
    if and only if $\Phi$ satisfies the $\omega$-NSP of order $s$, i.e.
    \begin{equation*}
\|\upsilon_{S}\|_{\omega,1}<\|\upsilon_{S^{C}}\|_{\omega,1}\;\mbox{for all}\;\upsilon\in\ker(\Phi)\;\mbox{and}\;S\subset[N]\;with\;\omega(S)=\sum_{j\in S}\omega_j^2\leq s.
\end{equation*}

 \end{example}

 \begin{example}
 \label{e01}
 In \cite{fried2012}\cite{ge2021}\cite{ge2021weighted}\cite{zhou2013}, for the weight vector $\boldsymbol{\omega}\in(0,1]^N$, given a prior support estimate $T_0\subset [N]$, $\boldsymbol{\omega}$ is defined by
\begin{equation}
\label{weight}
    \omega_i=\begin{cases}
  \gamma& \text{ if } i\in T_0 \\
  1& \text{ if } i\in T_0^C
\end{cases}
\end{equation}
for some fixed $\gamma\in(0,1)$. 

The sparsity of the recovered vector is required to be $\|x\|_0=card(S)$, where $S$ is the support of vector $x$. It is equivalent to the standard sparsity, i.e. $|S|\leq s$. So $\nu_{\omega}(S)=|S|$.

In this case, Theorem \ref{ewl1} can be described as follow. Every s-sparse vector $\hat{x}$ is the unique solution of 
\begin{equation*}
    \min\|x\|_{\omega,1}\;\text{subject to }\Phi x=b
\end{equation*}
with $b=\Phi\hat{x}$, if and only if $\Phi$ satisfies $\omega$-NSP with order $s$, i.e.
    \begin{equation*}
\|\upsilon_{S}\|_{\omega,1}<\|\upsilon_{S^{C}}\|_{\omega,1}\;\mbox{for all}\;\upsilon\in\ker(\Phi)\;\mbox{and}\;S\subset[N]\;with\;|S|<s.
\end{equation*}

\end{example}

\begin{example}
    Assuming that the weights can take any real value within the range of $(0,+\infty)$, there are no specific constraints or restrictions on the values of the weights. There is currently insufficient evidence to define a specific notion of weighted sparsity for this type of weights. Nevertheless, the theorems for $\omega$-NSP still hold.
   If we temporarily define the weighted sparsity in the same manner as the standard sparsity, where $\nu_{\omega}(S)=|S|\leq s$, the result can be described as follows:

    Every s-sparse vector $\hat{x}$ is the unique solution of 
\begin{equation*}
    \min\|x\|_{\omega,1}\;\text{subject to }\Phi x=b
\end{equation*}
with $b=\Phi\hat{x}$, if and only if $\Phi$ satisfies $\omega$-NSP with order $s$, i.e.
    \begin{equation*}
\|\upsilon_{S}\|_{\omega,1}<\|\upsilon_{S^{C}}\|_{\omega,1}\;\mbox{for all}\;\upsilon\in\ker(\Phi)\;\mbox{and}\;S\subset[N]\;with\;|S|\leq s.
\end{equation*}

Theoretically, $\nu_{\omega}(S)$ could take any form as a constraint on the set. In the examples mentioned above, $\nu_{\omega}(S)$ essentially takes the form of the $\ell_0$ norm constraint. It is also possible for $\nu_{\omega}(S)$ to take the form of other norms such as $\ell_1$ or $\ell_2$, in both weighted and unweighted cases. 
However, the no-constraint weights may not have corresponding specific application scenarios in practical applications.
\end{example}

\section{Robust recovery of compressible signals}
\label{section3}
Similar to $\omega$-NSP discussed earlier, a version of the generalized weighted restricted isometry property exists that applies to types of weights. 
\begin{definition}[$\omega$-RIP]
    \label{wrip}
     For given weights $\omega$, define the $\omega$-RIP constant $\delta_{\omega,s}$ associated with the matrix $A\in \mathbb{C}^{m\times{N}}$ as the smallest number such that 
    \begin{equation*}
        (1-\delta_{\omega,s})\|x\|_{2}^{2}\le\|Ax\|_{2}^2\le(1+\delta_{\omega,s})\|x\|_{2}^{2}
    \end{equation*}
    for arbitrary vector with its support set $S$ satisfying $\nu_{\omega}(S)\leq s$, which is the same as the one in Definition \ref{ensp}.\\
    And we say that matrix $A$ satisfies $\omega$-RIP of order $s$ with the constant $\delta_{\omega,s}$.
\end{definition}

\begin{definition}[$\omega$-RIP-NSP]
\label{wripnsp}
A matrix $\Phi$ has $\omega$-RIP-NSP with $\delta_{\omega,s}$ if it has the same null space as a matrix that has $\omega$-RIP with $\delta_{\omega,s}$.
\end{definition}
According to the definitions, if $A$ satisfies $\omega$-RIP, then for any invertible matrix $U$, $\Phi=UA$ satisfies $\omega$-RIP-NSP.

To describe the stability, the error of best weighted s-term approximation is involved. For the compressible vectors which are not exactly sparse, stability ensures that the results are not highly sensitive to minor variations, leading to more reliable behavior.

For $s\geq1$, the error of best weighted s-term approximation of the vector $\boldsymbol{x}\in\mathbb{C}^N$ is defined as
\begin{equation}
\label{sigmax}    \sigma_s(\boldsymbol{x})_{\omega,1}=\inf_{\|\boldsymbol{z}\|_{\omega,0}\leq s}\|\boldsymbol{x-z}\|_{\omega,1}.
\end{equation}

Unlike the error of best s-term approximation, which is defined as $\sigma_s(\boldsymbol{x})=\|\boldsymbol{x}_{S^c}\|$, where the set $S$ includes the largest s entries of the absolute value of  $x$, $\sigma_s(\boldsymbol{x})_{\omega,1}$ is not straightforward in general. There often exists a subset S such that $\sigma_s(\boldsymbol{x})_{\omega,1}=\|\boldsymbol{x}-\boldsymbol{x}_S\|_{\omega,1}=\|\boldsymbol{x}_{S^c}\|_{\omega,1}$, since $[N]$ is finite.

With the certain $\nu_{\omega}(S)$, if $\omega$-RIP can provide stable and robust error bounds, then the sensing matrix satisfying $\omega$-RIP-NSP will also guarantee a stable and robust reconstruction. However, the constants may have some adjustments when transitioning from $\omega$-RIP to $\omega$-RIP-NSP.

We call that $\omega$-RIP is \textit{stable and robust}, if for any matrix $A\in\mathbb{C}^{m\times N}$ satisfies $\omega$-RIP of order  $ts$  with $\delta_{\omega,tk}$ for some positive integer $t$ and $s$, $x \in \mathbb{C}^{N}$, $y=A x+e \in \mathbb{C}^{m}$ with $\|e\|_{2} \leq \rho$, the solution $\hat{x}$ of
\begin{equation}
\label{al1}
    \min_{z\in\mathbb{C}^{N}}\|z\|_{\omega,1}\;subject\;to\;\|Az-y\|_{2}\le\rho
\end{equation}
satisfies
\begin{align}
\label{c4}
             \left \| x-\hat{x}  \right \| _{\omega,1}\le C_{1}\sigma_{s}(x)_{\omega,1}+D_{1}\sqrt{s}\rho;\\
             \left \| x-\hat{x}  \right \| _{2}\le C_{2}\frac{\sigma_{s}(x)_{\omega,1}}{\sqrt{s}}+D_{2}\rho.
          \end{align}  
where $C_ {1},\;D_{1},\;C_{2},\;D_{2}$ depend on $k$ and $ \delta_{\omega,tk}$, and $\sigma_{k}(x)_{\omega,1}$ is the error of best weighted k-term approximation of the vector $x\in\mathbb{C}^N$.
\begin{proposition}
    \label{probd}
    Suppose  $\omega$-RIP is \textit{stable and robust} as defined above, then for the matrix $\Phi\in\mathbb{C}^{m\times{N}}$ that satisfies $\omega$-RIP-NSP with $\delta_ {\omega,ts}$ for some $t$ and $s$, $x\in\mathbb{C}^{N}$, $\hat{x}\in\mathbb{C}^{N}$ such that $\|\hat{x}\|_{\omega,1}\leq \|x\|_{\omega,1}$ and $\|\Phi\hat{x}-\Phi x\|_2\leq2\epsilon$.
 $x$ and $\hat{x}$ satisfies
    \begin{align}
    \label{a1}
     \left \| x-\hat{x}  \right \| _{\omega,1}\le C_{1}\sigma_{s}(x)_{\omega,1}+D_{1}\sqrt{s}\frac{\sqrt{1+\delta_{\omega,ts}}\sqrt{N_\nu}\epsilon}{\lambda(\Phi)};\\
    \left \| x-\hat{x}  \right \| _{2}\le C_{2}\frac{\sigma_{s}(x)_{\omega,1}}{\sqrt{s}}+D_{2}\frac{\sqrt{1+\delta_{\omega,ts}}\sqrt{N_\nu}\epsilon}{\lambda(\Phi)}.
    \end{align}
where $\lambda(\Phi)$ is the largest positive singular value of $\Phi$,  $C_{1},\;D_{1},\;C_{2},\;D_{2}$ depend on $s$ and $ \delta_{\omega,ts}$,  $\delta_ {\omega,ts}$ is the $\omega$-RIP-NSP constants (see in Definition \ref{wripnsp}), $\sigma_{s}(x)_{\omega,1}$ is the error of best weighted s-term approximation of the vector $x\in\mathbb{C}^N$, and $N_\nu$ is the number of partitions of $[N]$ defined as \ref{op}.
\end{proposition}

It is obvious that the conditions $\|\hat{x}\|_{\omega,1}\leq \|x\|_{\omega,1}$ and $\|\Phi\hat{x}-\Phi x\|_2\leq2\epsilon$ is more general than $\hat{x}$ is the solution of \eqref{phi1} blow. So Theorem can be regarded as a corollary of Proposition \ref{probd}.
\begin{theorem}
\label{thmbd}
    Suppose  $\omega$-RIP is \textit{stable and robust} as defined above, then for the matrix $\Phi\in\mathbb{C}^{m\times{N}}$ that satisfies $\omega$-RIP-NSP with $\delta_ {\omega,ts}$ for some $t$ and $s$, $x\in\mathbb{C}^{N}$, $\tilde{y}=\Phi x+\tilde{e} \in \mathbb{C}^{m}$ with $\|\tilde{e}\|_{2} \leq \epsilon$, the solution $\hat{x}$ of 
     \begin{equation}
     \label{phi1}
    \min_{z\in\mathbb{C}^{N}}\|z\|_{\omega,1}\;subject\;to\;\|\Phi z-\tilde{y}\|_{2}\le\epsilon ,
      \end{equation}
 satisfies
    \begin{align}
     \left \| x-\hat{x}  \right \| _{\omega,1}\le C_{1}\sigma_{s}(x)_{\omega,1}+D_{1}\sqrt{s}\frac{\sqrt{1+\delta_{\omega,ts}}\sqrt{N_\nu}\epsilon}{\lambda(\Phi)};\\
    \left \| x-\hat{x}  \right \| _{2}\le C_{2}\frac{\sigma_{s}(x)_{\omega,1}}{\sqrt{s}}+D_{2}\frac{\sqrt{1+\delta_{\omega,ts}}\sqrt{N_\nu}\epsilon}{\lambda(\Phi)}.
    \end{align}
where $\lambda(\Phi)$ is the largest positive singular value of $\Phi$,  $C_{1},\;D_{1},\;C_{2},\;D_{2}$ depend on $s$ and $ \delta_{\omega,ts}$,  $\delta_ {\omega,ts}$ is the $\omega$-RIP-NSP constants (see in Definition \ref{wripnsp}), $\sigma_{s}(x)_{\omega,1}$ is the error of best weighted s-term approximation of the vector $x\in\mathbb{C}^N$, and $N_\nu$ is the number of partitions of $[N]$ defined as \ref{op}.
\end{theorem}

The former term of the bound, i.e. the term including $\sigma_s(x)_{\omega,1}$, represents the stability of the reconstruction. It ensures that the reconstructed vector is close to the original vector even if the original vector is not strictly sparse. The latter term, i.e. the term involved with $\epsilon$, represents the robust part of the bound, which characterizes the ability of the reconstruction method to handle noise and errors in the measurements.  And the bounds in Theorem \ref{thmbd} will not be as good as the bounds from $\omega$-RIP, since $\omega$-RIP-NSP can not constrain the stretching of the matrix. There still exist difficulties in verifying $\omega$-RIP-NSP since it is hard to analyze the null space. But it has applications in some practical cases, such as in compressive imaging (wavelet sparse images, Fourier measurements)\cite{krahmer2013}.

Before proving Theorem \ref{probd}, we will introduce a lemma that will be utilized in the proof.

For the matrix $A\in\mathbb{C}^{m\times N}$, the \textit{operator norm} of $A$ is defined as
$$
    \|A\|_{op}:=\max_{\|x\|_2\leq1}\|Ax\|_2=\max_{\|x\|_2=1}\|Ax\|_2.
$$

Separate the index set $\{1,2,\ldots,N\}$ into $\{1,2,\ldots,k_1\}\cup\{k_1+1,\ldots,k_2\}\cup\cdots\cup\{k_{N_\nu-1}+1,\ldots,N\}=\bigcup_ {i=1}^{N_\nu}T_{i}$ respectively, where $N_\nu$ denotes the number of subsets and for any $i<N_\nu$, $T_i$ satisfying $\nu_\omega(T_i)\leq s$ and $\nu_\omega(T_i\cup\{k_i+1\})>s$.

In the special case $\nu_\omega(T_i)=\sum_{j\in T_i}\omega_j^2$, there is an estimation of $N_\nu$.
According to the definition of $N_\nu$, we have $card(T_i)+1\geq\frac{s}{\|\omega\|_\infty^2}$. So 
\begin{align}
\label{nv}
    N_\nu\leq\frac{N}{\min_{i=1,\ldots,N_\nu} card(T_i)}\leq\frac{N\|\omega\|_\infty^2}{s}+1.
\end{align}
\begin{lemma}\label{op}

  For the given weight $\omega$, and $s\geq2\max_{i\in[N]}\nu_{\omega}(\{i\})$, if the matrix $\Phi\in\mathbb{C}^{m\times N}$ satisfies $\omega$-RIP with constant $\delta_{\omega,s}$, then the operator norm of $\Phi$ satisfies
  $$\|\Phi\|_{op}\le\sqrt{N_\nu}\sqrt{1+\delta_{\omega,s}}.$$
\end{lemma}
  
  \begin{proof}
      The notation $\Phi_{T_i}$ represents the restriction of matrix $\Phi$ to the index set $T_i$. This restriction is defined as the submatrix of $\Phi$ formed by selecting only the columns indexed by $T_i$. And we denote the $j_{th}$ column of $\Phi$ as $\Phi_j$.
      
    According to the definition of the operator norm of matrices,
    \begin{equation*}
      \begin{aligned}
      \|\Phi\|_{op}&=\max_{\|x\|_{2}=1}\|\Phi x\|_{2}=\max_{\|x\|_{2}=1}\sum_{j=1}^{N}\|\Phi_j x_j\|_2\\
      &=\sum_{i\geq1}\max_{\|x\|_{2}=1}\|\Phi_{T_{i}}x_{T_{i}}\|_{2}
      \leq\sum_{i\geq1}\max_{\|x\|_{2}=1}\|\Phi x_{T_{i}}\|_{2}\\
      &\le\sum_{i\geq1}\sqrt{1+\delta
      _{{\omega,s}}}\max_{\|x\|_{2}=1}\|x_{T_{i}}\|_{2}\\
      &\le\sqrt{1+\delta_{\omega,s}}\max_{\|x\|_{2}=1}\sqrt{\sum_{i\geq1}
      \|x_{T_{i}}\|_{2}^2}\sqrt{\sum_{i\geq1}1}\\
      &\leq\sqrt{1+\delta_{\omega,s}}\sqrt{N_\nu}.
      \end{aligned}
    \end{equation*}
  where the $\sum_{i\geq1}1$ is the number of subsets. The first inequality arises from the right side of the $\omega$-RIP on $\Phi$.
  \end{proof}

  \begin{proof}[Proof of Theorem \ref{probd}]
    For the sensing matrix $A$ satisfies $\omega$-RIP with the constant $\delta_ {\omega,ts}$, there exists an invertible matrix $U$ such that $\Phi=UA$.  $\lambda(U)$ is the largest positive singular value of $U$.

Suppose the vector $\hat{x}$ is the solution of \eqref{al1}, then $\hat{x}$ satisfies $\|\hat{x}\|_{\omega,1}\leq \|x\|_{\omega,1}$ and $\|A\hat{x}-A x\|_2\leq2\rho$ for any $x\in\mathbb{C}^N$.

$$
    \|\Phi\hat{x}-\Phi x\|_{2}=\|UA\hat{x}-UAx\|_{2}
    \le\lambda(U)\|A\hat{x}-A x\|_{2}
    \le 2\lambda(U)\rho.
    $$
Let $\rho=\frac{\epsilon}{\lambda(U)}$, we have $\|\Phi\hat{x}-\Phi x\|_{2}\leq2\epsilon$. $x$ and $\hat{x}$ also satisfy $\|\hat{x}\|_{\omega,1}\leq \|x\|_{\omega,1}$.

    By the assumption that $\omega$-RIP is \textit{stable and robust}, we have
    \begin{align}\label{d}
           \left \| x-\hat{x}  \right \| _{\omega,1}\le C_{1}\sigma_{s}(x)_{\omega,1}+D_{1}\sqrt{s}\frac{\epsilon}{\lambda(U)};\\
           \left \| x-\hat{x}  \right \| _{2}\le C_{2}\frac{\sigma_{s}(x)_{\omega,1}}{\sqrt{s}}+D_{2}\frac{\epsilon}{\lambda(U)}.
        \end{align}\\
    Then for $\lambda(U)$, based on the definition of singular value,
    \begin{align*}
    \lambda^{2}(\Phi)&=\max_{\|x\|_{2}=1}\|\Phi^* x\|_2\\
    &=\max_{\|x\|_{2}=1}\langle UAA^{\ast}U^{\ast}x,x\rangle\\
    &\le\|AA^{\ast}\| _{op}\max_{\|x\|_{2}=1}\|U^{\ast}x\|^{2}\\
    &=\|A\|_{op}^{2}\lambda^{2}(U)
    \end{align*}
    With Lemma \ref{op}, implies
    $$
    \lambda(U)\geq\frac{\lambda(\Phi)}{\sqrt{1+\delta_{\omega,s}}\sqrt{N_\nu}}
    $$
    By substituting it into \eqref{d}, we obtain error bounds in \eqref{a1}.
\end{proof}

Now we will exhibit that $\omega$-RIP can provide the error bounds in \eqref{c4} form for the two special types of weights. In the case of the weighted $\ell_1$ minimization problem \eqref{omegal1} with noise, the weighted restricted isometry property ($\omega$-RIP) is a sufficient condition for a robust and stable recovery.

\subsection{Case 1: $\nu_{\omega}(S)=\sum_{i\in S}\omega_i^2$ with $\|\omega\|_\infty\geq1$}
\label{ex1}
In the work of \cite[Theorem 4.5]{rauhut2016} , it was proven that $\omega$-RIP implies $\omega$-robust-NSP (as defined in \ref{robustnsp}). Furthermore, $\omega$-robust-NSP provides stable and robust error bounds (see in \cite[Theorem 4.2]{rauhut2016}). In conclusion, $\omega$-RIP can provide stable and robust error bounds, which is required in Theorem \ref{thmbd}. Then the result for this type of weights can be inferred. According to \cite{rauhut2016}, this type of weights is required that $\|\omega\|_\infty\geq1$.

    For the matrix $\Phi\in\mathbb{C}^{m\times{N}}$ that satisfies $\omega$-RIP-NSP with $\delta_ {\omega,3s}<1/3$, and $x\in\mathbb{C}^{N}$ and $s\geq2\|\omega\|_{\infty}^2$, $\tilde{y}=\Phi x+\tilde{e} \in \mathbb{C}^{N}$ with $\|\tilde{e}\|_{2} \leq \epsilon$, $\omega$-RIP is \textit{stable and robust}, and the solution $\hat{x}$ of \eqref{phi1} satisfies
    \begin{align*}
     \left \| x-\hat{x}  \right \| _{\omega,1}\le C_{1}\sigma_{s}(x)_{\omega,1}+D_{1}\sqrt{s}\frac{\sqrt{1+\delta_{\omega,3s}}\sqrt{N_\nu}\epsilon}{\lambda(\Phi)};\\
    \left \| x-\hat{x}  \right \| _{2}\le C_{2}\frac{\sigma_{s}(x)_{\omega,1}}{\sqrt{s}}+D_{2}\frac{\sqrt{1+\delta_{\omega,3s}}\sqrt{N_\nu}\epsilon}{\lambda(\Phi)}.
    \end{align*}
   
    where $C_ {1},D_{1},C_{2},D_{2}$ only depends on $\delta_{2s}$.

   This result is new even in this special case.

\subsection{Case 2: $\nu_{\omega}(S)=|S|$ with $\|\omega\|_\infty\leq1$}
\label{ex2}
For $\nu_{\omega}(S)=|S|$ case, several studies have attempted to improve the bound of the restricted isometry constants. And with the same sparse function as unweighted case, $\omega$-RIP is equivalent to the standard RIP. However, a sharp bound has yet to be established. In \cite{ge2021}, Ge et al. provided proof that $\omega$-RIP offers stable and robust error bounds with some $\delta_{ts}$ for certain vectors (see in  \cite[Theorem 3.1]{ge2021}), but the special structured weights are required, which is defined as \eqref{weight} with prior information support and $\|\omega\|_\infty\leq1$.

Without the prior information support, there still exists the result about RIP implies stable and robust error bounds, with some difference between the constants.

Before proving the theorem, there is a lemma required in the proof, which is from \cite{rauhut2016}.
\begin{lemma}[see Proposition 6.3 in \cite{rauhut2016}]
\label{dst}
    If $u,\;v\in\mathbb{C}^N$ such that $\|u\|_0\leq s$, $\|v\|_0\leq t$, $supp(u)\cap supp(v)=\emptyset$ and the matrix $A$ satisfies RIP of order $s+t$, then $|\langle Au,Av\rangle |\leq \delta_{s+t}\|u\|_2\|v\|_2$. 
\end{lemma}

\begin{theorem}
\label{01th}
    For any given $\gamma<1$ and the weight $\omega=[\omega_1,\ldots,\omega_N]^T\in [\gamma,1]^N$, if the matrix $A$ satisfies RIP of order $2s$ with $\delta_{2s}<\frac{\gamma}{\gamma+2}$, then the vector $x\in\mathbb{C}^N$ and the solution of \eqref{al1} satisfy
    \begin{align*}
     \left \| x-\hat{x}  \right \| _{\omega,1}\le A_{1}\sigma_{s}(x)_{\omega,1}+B_{1}\sqrt{s}\rho;\\
    \left \| x-\hat{x}  \right \| _{2}\le A_{2}\frac{\sigma_{s}(x)_{\omega,1}}{\sqrt{s}}+B_{2}\rho.
    \end{align*}
where $A_ {1},B_{1},A_{2},B_{2}$ only depends on $\delta_{2s}$ and $\gamma$.

    Additionally, $A$ also satisfies $\omega$-NSP of order $s$ with $\frac{\delta_{2s}}{\gamma-(\gamma+1)\delta_{2s}}$, which guarantees the exact recovery.
\end{theorem}
\begin{proof}
    Assume $\hat{x}=x+h$, then $h$ is the noise. While $x$ can be any vector in $\mathbb{C}^N$, $h$ is also an arbitrary vector.
    It is enough to consider an index set $S =: S_0$ of $s$ largest absolute entries of $h\cdot\omega$. And we partition $S_0^c$ as $S_1\cup S_2\cup\ldots$, where $S_l$ is the index set of $s$ largest absolute entries of $h\cdot\omega$ in $S_{l-1}^c$, $l\geq1$. So that $|h_j|\omega_j\leq|h_k|\omega_k$ for all $j\in S_l$, $k\in S_{l-1}$, $l\geq1$.

     According to the definition of $\ell_{\omega,1}$ norm, 
    $$
      \|h_{S}\|_{\omega,1}=\sum_{j\in S}|h_{j}|\omega_{j}\leq\sqrt{\sum_{j\in S}|h_{j}|^2}\sqrt{\sum_{j\in S}\omega_{j}^2}\leq\sqrt{s}\|h_{S}\|_{2}.
    $$
   $$
    \|Ah_S\|_2^2=\langle Ah_S,Ah\rangle-\langle Ah_S,A\left(\sum_{l\geq1}h_{S_l}\right)\rangle.
    $$
    By the $\omega$-RIP on $S$,
    $$
    \langle Ah_S,Ah\rangle\leq\|Ah_S\|_2\|Ah\|_2\leq\sqrt{1+\delta_{2s}}\|h_S\|_2\|Ah\|_2.
    $$
     When $l\geq1$, $S$ and $S_l$ are disjoint. By Lemma \ref{dst}, if the matrix $A$ satisfies $\omega$-RIP of order $2s$, then
         $$
|\langle Ah_S,Ah_{S_l}\rangle|\leq \delta_{2s}\|h_S\|_2\|h_{S_l}\|_2.
    $$
    Combining the equations above and with RIP on $S$, we have
     $$
      (1-\delta_{2s})\|h_S\|_2^2\leq\|Ah_
      S\|_2^2\leq\sqrt{1+\delta_{2s}}\|h_S\|_2\|Ah\|_2+\delta_{2s}\|h_
      S\|_2\left(\sum_{l\geq1}\|h_{S_l}\|_2\right).
    $$
    Divide both sides by$\|h_S\|_2$,
    \begin{equation}
    \label{l2tl2}
        \|h_S\|_2\leq\frac{\sqrt{1+\delta_{2s}}}{1-\delta_{2s}}\|Ah\|_2+\frac{\delta_{2s}}{1-\delta_{2s}}\sum_{l\geq1}\|h_{S_l}\|_2.
    \end{equation}
    For $i\in S_l$, $l\geq1$, we have $h_i\omega_i\leq\sum_{j\in S_{l-1}}\frac{1}{s}h_j\omega_j$.
  This, together with $\omega_i\in[\gamma,1],\;i=1,\ldots,N$, further implies that, for any $l\geq 1$, 
    \begin{align*}
        \|h_{S_l}\|_2^2=\sum_{k\in S_l}(|h_k|\omega_k)^2\omega_k^{-2}\leq\frac{1}{s^2}\|h_{S_{l-1}}\|_{\omega,1}^2\sum_{k\in S_l}\omega_k^{-2}
      \leq\frac{1}{\gamma^2 s}\|h_{S_{l-1}}\|_{\omega,1}^2;\\
      \|h_{S_l}\|_2\leq\frac{1}{\gamma\sqrt{s}}\|h_{S_{l-1}}\|_{\omega,1}.
    \end{align*}
    
  Substitute $\|h_{S_l}\|_2$ in \eqref{l2tl2} with the above estimate, we have  
    \begin{equation}
    \label{hs2}
        \|h_S\|_2\leq\frac{\sqrt{1+\delta_{2s}}}{1-\delta_{2s}}\|Ah\|_2+\frac{\delta_{2s}}{\gamma\sqrt{s}(1-\delta_{2s})}\|h\|_{\omega,1}.
    \end{equation}

    If $h\in\ker(A)$ and $\delta_{2s}<\frac{\gamma}{\gamma+2}$, then we have 
  $$
    \|h_{S}\|_{\omega,1}\leq\sqrt{s}\|h_S\|_2\leq\frac{\delta_{2s}}{\gamma-(1+\gamma)\delta_{2s}}\|h_{S^c}\|_{\omega,1}<\|h_{S^c}\|_{\omega,1}.
 $$

 So $A$ satisfies $\omega$-NSP of order $s$.

 Since $S$ is the index set of $s$ largest absolute entries of $h\cdot\omega$,  $\sigma_{s}(x)_{\omega,1}=\|x-x_{S}\|_{\omega,1}=\|x_{S^{c}}\|_{\omega,1}$.
  By the triangle inequality of the norms,
    \begin{align*}
      \|x\|_{\omega,1}+\|(x-\hat{x})_{S^{c}}\|_{\omega,1}&\leq 2\|x_{S^{c}}\|_{\omega,1}+\|\hat{x}_{S^{c}}\|_{\omega,1}+\|x_{S}\|_{\omega,1}\\
      &\leq 2\sigma_{s}(x)_{\omega,1}+\|(x-\hat{x})_{S}\|_{\omega,1}+\|\hat{x}\|_{\omega,1}.
    \end{align*}
    Since $h=\hat{x}-x$ and $\|\hat{x}\|_{\omega,1}\leq\|x\|_{\omega,1}$, 
    $$
      \|h_{S^{c}}\|_{\omega,1}\leq 2\sigma_{s}(x)_{\omega,1}+\|\hat{x}\|_{\omega,1}-\|x\|_{\omega,1}+\|h_{S}\|_{\omega,1}
      \leq 2\sigma_{s}(x)_{\omega,1}+\|h_{S}\|_{\omega,1}.
    $$
    So the error bound in $\ell_1$ norm estimates as 
\begin{align*}
        \|h_{S}\|_{\omega,1}\leq\sqrt{s}\|h_S\|_2
        \leq\frac{\sqrt{1+\delta_{2s}}}{1-\delta_{2s}}\sqrt{s}\|Ah\|_2+\frac{\delta_{2s}}{\gamma(1-\delta_{2s})}\|h\|_{\omega,1};\\
        \label{hsw1}
        \|h_{S}\|_{\omega,1}\leq\frac{2\gamma\sqrt{1+\delta_{2s}}}{\gamma(1-\delta_{2s})-2\delta_{2s}}\sqrt{s}\rho+\frac{2\delta_{2s}}{\gamma(1-\delta_{2s})-2\delta_{2s}}\sigma_{s}(x)_{\omega,1};\\
        \|h\|_{\omega,1}\leq\frac{4\gamma\sqrt{1+\delta_{2s}}}{\gamma(1-\delta_{2s})-2\delta_{2s}}\sqrt{s}\rho+\frac{2\gamma(1-\delta_{2s})}{\gamma(1-\delta_{2s})-2\delta_{2s}}\sigma_{s}(x)_{\omega,1}.
\end{align*}
So $A_1=\frac{2\gamma(1-\delta_{2s})}{\gamma(1-\delta_{2s})-2\delta_{2s}}$ and $B_1=\frac{4\gamma\sqrt{1+\delta_{2s}}}{\gamma(1-\delta_{2s})-2\delta_{2s}}$.

Then the error bound in $\ell_2$ norm is 
\begin{align*}
    \|h\|_2&\leq\|h_S\|_2+\|h_{S^c}\|_2\\
    &\leq\frac{\sqrt{1+\delta_{2s}}}{1-\delta_{2s}}\|Ah\|_2+\frac{\delta_{2s}}{\gamma\sqrt{s}(1-\delta_{2s})}\|h\|_{\omega,1}+\frac{1}{\gamma\sqrt{s}}\|h\|_{\omega,1}\\
    &\leq\frac{2\sqrt{1+\delta_{2s}}}{1-\delta_{2s}}\rho+\frac{1}{\gamma\sqrt{s}(1-\delta_{2s})}\left[\frac{4\gamma\sqrt{1+\delta_{2s}}}{\gamma(1-\delta_{2s})-2\delta_{2s}}\sqrt{s}\rho+\frac{2\gamma(1-\delta_{2s})}{\gamma(1-\delta_{2s})-2\delta_{2s}}\sigma_{s}(x)_{\omega,1}\right]\\
    &\leq\frac{(2\sqrt{1+\delta_{2s}})(\gamma(1-\delta_{2s})-2\delta_{2s}+2)}{(1-\delta_{2s})(\gamma(1-\delta_{2s})-2\delta_{2s})}\rho+\frac{2}{\gamma(1-\delta_{2s})-2\delta_{2s}}\frac{\sigma_{s}(x)_{\omega,1}}{\sqrt{s}}.
\end{align*}
So $A_2=\frac{2}{\gamma(1-\delta_{2s})-2\delta_{2s}}$ and 
\begin{equation}
\label{b2}
  B_2=\frac{(2\sqrt{1+\delta_{2s}})(\gamma(1-\delta_{2s})-2\delta_{2s}+2)}{(1-\delta_{2s})(\gamma(1-\delta_{2s})-2\delta_{2s})}.  
\end{equation}

\end{proof}

 For the matrix $\Phi\in\mathbb{C}^{m\times{N}}$ that satisfies $\omega$-RIP-NSP with $\delta_{2s}<\frac{\gamma}{\gamma+2}$, and $x\in\mathbb{C}^{N}$. $\tilde{y}=\Phi x+\tilde{e} \in \mathbb{C}^{N}$ with $\|\tilde{e}\|_{2} \leq \epsilon$, $\omega$-RIP is \textit{stable and robust}, and the solution $\hat{x}$ of \eqref{phi1} satisfies
    \begin{align*}
     \left \| x-\hat{x}  \right \| _{\omega,1}\le C_{1}\sigma_{s}(x)_{\omega,1}+D_{1}\sqrt{s}\frac{\sqrt{1+\delta_{2s}}\sqrt{\frac{N}{s}+1}\epsilon}{\lambda(\Phi)};\\
    \left \| x-\hat{x}  \right \| _{2}\le C_{2}\frac{\sigma_{s}(x)_{\omega,1}}{\sqrt{s}}+D_{2}\frac{\sqrt{1+\delta_{2s}}\sqrt{\frac{N}{s}+1}\epsilon}{\lambda(\Phi)}.
    \end{align*}
   
    where $C_ {1},D_{1},C_{2},D_{2}$ only depends on $\delta_{2s}$ and $\gamma$.

   This result is new even in this special case.
\section{ The interrelation between properties}
\label{section4}
This section discusses the interrelation between three key concepts: $\omega$-RIP, $\omega$-RIP-NSP, and $\omega$-NSP. 
Whether a matrix satisfies the weighted properties above involves the constants of the properties. When we claim that a matrix satisfies certain weighted properties, it implies that the matrix satisfies those properties with specific constants associated with them.

It is obvious that $\omega$-RIP is strictly stronger than $\omega$-RIP-NSP by the definition of $\omega$-RIP-NSP. The set of matrices  $\{U \Phi: \text{U  is invertible}  \}$  share the same null space but can have dramatically different RIP constants. This is because $\omega$-RIP-NSP can not effectively constrain the stretching of the matrix. So there must exist matrices that share the same kernel with the matrix satisfying $\omega$-RIP with a certain $\delta$, however, the restricted isometry constants of these matrices can be much larger or smaller than $\delta$.
\begin{proposition}
\label{rrn}
    Suppose the matrix $\Phi$ satisfies $\omega$-RIP of order $s$ with $\delta$, then there exists an invertible matrix $U$ such that $U\Phi$ does not satisfy $\omega$-RIP of order $s$ with $\delta$, while $U\Phi$ satisfying $\omega$-RIP-NSP of order $s$ with $\delta$.
\end{proposition}
\begin{proof}
    Assume the largest positive singular value of $U$ is $\lambda(U)<\frac{1-\delta}{1+\delta}$.
    Then for the vector $x\in\mathbb{C}^N$ with its support set $S$ satisfying $\nu_{\omega}(S)\leq s$,
    $$
        \|U\Phi x\|_2^2\leq\lambda(U)\|\Phi x\|_2^2.
    $$
    Since  the matrix $\Phi$ satisfies $\omega$-RIP of order $s$ with $\delta$, we have
    $$
    \|\Phi x\|_2^2\leq(1+\delta)\|x\|_2^2.
    $$
    Combine two inequations, we have 
    \begin{equation*}
         \|U\Phi x\|_2^2\leq\lambda(U)(1+\delta)\|x\|_2^2<(1-\delta)\|x\|_2^2,
    \end{equation*}
which violates $\omega$-RIP of order $s$ with $\delta$.
\end{proof}

It is necessary to focus on specific cases rather than studying them in a general setting, since the result that $\omega$-RIP implies $\omega$-NSP strongly depends on the \textit{sparse function} $\nu_\omega(S)$. So in this section, we must consider the special cases that $\nu_{\omega}(S)=|S|$ and $\nu_{\omega}(S)=\sum_{i\in S}\omega_i^2$. We will show that in these two special cases, $\omega$-RIP is strictly stronger than $\omega$-RIP-NSP and $\omega$-RIP-NSP is also strictly stronger than $\omega$-NSP.

 It is necessary to establish $\omega$-RIP-NSP properly including $\omega$-NSP. This is achieved by demonstrating the existence of the matrices that satisfy $\omega$-NSP but fail to satisfy $\omega$-RIP-NSP as counter-examples. This result shows that $\omega$-RIP-NSP provides a stronger guarantee for robust and stable signal recovery, as compared to $\omega$-NSP. 
To construct counter-examples, the weights described in the examples of Section \ref{section2} are utilized in the construction of the required matrices. The two types of weights often correspond to specific structured matrices commonly used in practice.
\subsection{Case 1: $\nu_{\omega}(S)=\sum_{i\in S}\omega_i^2$}
The first type of weights is motivated by application to function interpolation, and it is common to construct interpolation polynomials using discrete orthonormal systems.  By considering these systems, we can establish a strictly inclusive relationship between $\omega$-RIP-NSP and $\omega$-NSP, demonstrating the improved guarantee offered by $\omega$-RIP-NSP over $\omega$-NSP.

Finding appropriate measurement matrices for compressive sensing is a crucial task.  In many cases, completely random matrices, such as subgaussian random matrices are used\cite{rudelson2008}\cite{jac2011}\cite{dirksen2018}. Additionally, the specifically structured matrices, which  are generated by a random choice of parameters, can also provide recovery guarantees\cite{j2014A}. In this section, we will exclusively study special cases where randomly sampling functions whose expansion in a bounded orthonormal system is sparse or compressible. The weighted discrete orthonormal system is the system used in this section.

For the unitary matrix $\mathbf{U}\in\mathbb{C}^{N\times{N}}$, the normalized column vectors$\sqrt{N}\mathbf{u}_{1},\,\sqrt{N}\mathbf{u}_{2},\ldots,\,\sqrt{N}\mathbf{u}_{N}$ which are orthogonal, form an orthonormal system with respect to the discrete uniform measure on $[N]$ given by $\nu(B)=card(B)/N\;for\;B\subset[N]$:
$$
\frac{1}{N}\sum_{t=1}^N \sqrt{N}\mathbf{u}_k(t) \sqrt{N}\,\overline{ \mathbf{u}_{\ell}(t)}=\left\langle\mathbf{u}_k, \mathbf{u}_{\ell}\right\rangle=\delta_{k, \ell}=\{\begin{matrix} 
 1,\;k=\ell  \\  
  0,\;k\ne\ell 
\end{matrix} \quad k,\;\ell \in[N].
$$

Choosing the points $t_{1},\ldots,t_{m}$ independently and uniformly at random from $[N]$ generates the random matrix $\mathbf{A}$, i.e. selecting m rows independently and uniformly at random from the original $\mathbf{U}$ to form $\mathbf{A}$, which can be described as:
$$
\mathbf{A}=\mathbf{R}_{T}\mathbf{U}
$$
where $T=\{t_{1},\ldots,t_{m}\}\subset [N]$, and $\mathbf{R}_{T}:\mathbb{C}^{N}\longrightarrow\mathbb{C}^{m}$ denote the random subsampling operator:
$$
\left(\mathbf{R}_T \mathbf{z}\right)_{\ell}=z_{t_{\ell}}, \quad \ell \in[m].
$$
where $z\in\mathbb{C}^{N}=[z_{1},z_{2},\ldots,\,z_{N}]^{T}$.\\

For any given weight $\omega\geq 1$, there exists a unitary matrix satisfying $\omega_{j}\geq\sqrt{N}\|\mathbf{u}_{j}\|_{\infty}$. A typical example of such a matrix is the discrete Fourier transform(DFT) with N dimensions, the infinity norm of any column vectors is $\frac{1}{\sqrt{N}}$, allowing the use of any weights.

To construct matrices satisfying $\omega$-NSP but not $\omega$-RIP-NSP, the initial step is to find the matrices satisfying $\omega$-NSP with this type of weights. According to the previous literature, certain structured matrices satisfy $\omega$-RIP under specific conditions, and since $\omega$-RIP implies $\omega$-NSP, these structured matrices can be used as a starting point.

Verifying whether matrices satisfy the $\omega$-RIP property can be challenging in practice. However, there are some structured matrices, such as Gaussian matrices, subgaussian matrices, and matrices from orthogonal systems, that have been shown to satisfy the $\omega$-RIP under certain conditions\cite{rauhut2016}\cite{j2014A}. 
Here is the result for the orthogonal system matrices that it will satisfy $\omega$-RIP with some probability if the number of measurements is large enough.
\begin{theorem}{see in \cite[Theorem 1.1]{rauhut2016}}
    \label{thmrip}
Fix parameters $\delta\in(0,1)$, let $\left(\psi_j\right)_{j \in \Lambda}$ be an orthogonal system of finite size $N=|\Lambda|$. Consider weights satisfying $\omega_j \geq\left\|\psi_j\right\|_{\infty}$. Fix
$$
m \geq C \delta^{-2} s \log ^3(s) \log (N)
$$
and suppose that $t_1, t_2, \ldots, t_m$ are drawn independently from the orthogonalization measure associated with the $\left(\psi_j\right)_{j\in\Lambda}$. Then the normalized sampling matrix $A\in \mathbb{C}^{m \times N}$ with entries $\boldsymbol{A}_{\ell, k}=\frac{1}{\sqrt{m}}\psi_k\left(t_{\ell}\right)$ satisfies the weighted restricted isometry property of order $s$ with the probability exceeding  $1-N^{-\log^3(s)}$, i.e. $\delta_{\omega, s} \leq \delta$.
\end{theorem}

There exists a connection between $\omega$-RIP and $\omega$-NSP. In the unweighted case, many studies have demonstrated that matrices satisfying the restricted isometry property also satisfy the null space property (e.g., Theorem 6.13 in \cite{j2014A}). Similar conclusions can be established in the weighted case, albeit with some differences in the values of the restricted isometry constant.
And Rauhut\cite{rauhut2016} has further established that under certain restrict isometry constants and order of $\omega$-RIP, $\omega$-RIP implies $\omega$-robust-NSP, which subsequently implies $\omega$-NSP. 
\begin{theorem}{see in \cite[Theorem 4.5]{rauhut2016}}
\label{1/3nsp}
Let matrix $A\in\mathbb{C}^{m\times{N}}$ with $\omega$-RIP constant
$$\delta_{\omega,3s}<1/3$$
for $s\geq2\|\omega\|_{\infty}^2$. Then $A$ satisfies $\omega$-NSP of order $s$ with constants $\gamma=2\delta_{\omega,3s}/(1-\delta_{\omega,3s})<1$.
\end{theorem}

Next, with the help of the weighted discrete orthonormal system, we will demonstrate that $\omega$-RIP-NSP is strictly stronger than $\omega$-NSP, by constructing a matrix that satisfies $\omega$-NSP but not $\omega-$RIP-NSP. Before proving the main conclusion, we first obtain a corollary for the discrete orthogonal system using the theorems mentioned above.
By substituting $s$ with $3s$ in Theorem \ref{thmrip}, and utilizing Theorem \ref{1/3nsp}, we can derive the following corollary.

\begin{corollary}
\label{c1}
    For any given positive constant $\gamma<1$, if constants $C$, $m,\,s,\,N$ satisfy $N/2> m\geq 3C(\frac{2+\gamma}{\gamma})^{2}s\log^{3}(3s)\log(N)$, there exists $R\in \mathbb{C}^{m\times{N}}$ with $\omega$-NSP of order $s$ with $\gamma$ and $(1,1,\ldots,1)^{T}\in ker(R)$.
\end{corollary}
    
\begin{proof}
      There exists a unitary matrix $\mathbf{U}\in\mathbb{C}^{N\times{N}}$ such that the first row vector is $(1,1,\ldots,1)^{T}$.

      For the given $\gamma$, let $\delta=\frac{\gamma}{2+\gamma}$, $R$ be the matrix sampled independently and uniformly at random from $U$. If the sample size satisfies $m \geq C \delta^{-2} 3s \log ^3(3s) \log (N)$ as in \ref{thmrip}, then $R$ satisfies $\omega$-RIP of order $3s$ with probability $p_{1}=1-N^{-\log^3(3s)}$, i.e. $\delta_{\omega, 3s} \leq\delta<1/3$.
      Then from Theorem \ref{1/3nsp}, $R$ satisfies $\omega$-NSP of order $s$ with $\tilde{\gamma}=\frac{2\delta_{\omega, 3s}}{1-\delta_{\omega, 3s}}<\gamma$ and probability $p_1$, which implies that $R$ satisfies $\omega$-NSP of order $s$ with $\gamma$ and probability $p_1$.

      Since the row vectors of unitary matrices are orthogonal, $R$ can not select the first row to allow $e=(1,1,\ldots,1)^{T}\in \ker(R)$. The probability of that is $p_{2}=1-\frac{m}{N}$.
      To guarantee there exists a matrix $R$ satisfies $\omega$-NSP of order $s$ and $e=(1,1,\ldots,1)^{T}\in \ker(R)$, the matrices satisfying $\omega$-NSP with $p_1$ can not all involve the first row, which means $p_1> 1-p_2$.
      Since $m\geq 3C(\frac{2+\gamma}{\gamma})^{2}s\log^{3}(3s)\log(N)$, we have $p_1\geq\frac{1}{2}$. And if $m<N/2$, $p_2>\frac{1}{2}$. So $p_1+p_2>1$ holds.
\end{proof}

Fix that the weighted restricted isometry constant of order $3s$ is $\delta_{\omega,3s}$, then $\omega$-NSP of order $s$ holds for that matrix. And the matrices that share the same null space with the $\omega$-RIP matrix satisfy $\omega$-RIP-NSP with $\delta_{\omega,3s}$. Since $\omega$-NSP only depends on the null space, those matrices also satisfy $\omega$-NSP. In conclusion, $\omega$-RIP-NSP implies $\omega$-NSP. If there exists a matrix that shares the same null space with the $\omega$-RIP matrix, but with the specific constant $\delta_{\omega,3s}$, $\omega$-RIP-NSP does not hold, the result can be shown.

Based on Corollary \ref{c1}, the conditions for the parameters are determined. As a result, we can establish the following conclusion.
\begin{theorem}
      \label{1.3}
      For the given weight $\omega$, if $m,\;s,\;N\in\mathbb{Z}$ satisfy that $N/2\geq m\geq 81Cs\log^{3}(3s)\log(N)$, $N\geq24\|\omega\|_\infty^2s$ and $s>23040\|\omega\|_\infty^{6}$, then there exists a sensing matrix $\Phi\in\mathbb{C}^{m\times{N}}$ that satisfies $\omega$-NSP of order $s$, but does not satisfy $\omega$-RIP-NSP of order $3s$ with $\delta_{\omega,3s}<1/3$. 
\end{theorem}

 \begin{proof}
$\omega$-NSP and $\omega$-RIP-NSP only depend on the null space of matrices. In the proof, firstly we get a matrix satisfying $\omega$-NSP from Corollary \ref{c1}. And we construct a matrix depending on the null space of the former matrix. The second step of the proof is to show that the matrix we construct satisfies $\omega$-NSP. The third step is to prove that the matrix violates robustness and stability so it does not satisfy $\omega$-RIP-NSP.

      (1) construct the matrix

      For the given weight $\omega=[\omega_1,\ldots,\omega_N]^T$, $\omega_j\geq1$, $j=1,\ldots,N$, $k$ is a positive integer smaller than $s$ such that $\sum_{i=1}^k\omega_i^2\leq s$ and $\sum_{i=1}^{k+1}\omega_i^2>s$.
      By the assumption that $m\geq 81Cs\log(N)\log^{3}(3s)$ and $k\leq s$, it is obvious that $k\leq27Cs\log^3(3s)[3\log(N)-\log(N-s)]$, so we have $(m-k)\geq 27Cs\log(N-k)\log^{3}(3s)$.
      By Corollary \ref{c1}, for weight $\tilde{\omega}=[\omega_{k+1},\ldots,\omega_N]^T$, there exists a matrix $A\in \mathbb{R}^{(m-k)\times (N-k)}$ satisfies $\omega$-NSP of order $s$ with $\gamma=1/3$ and $e=(1,1,\ldots,1)^{T}\in \ker(A)$.

      Define $$\mathcal{N}_{e}=\{x\in \ker (A): x\bot e\}.$$
      Since $A$ is sampled by a unitary matrix, $rank(A)=m-k$, so dimension of $\ker(A)$ is $N-m$, $\dim(\mathcal{N}_{e})=N-m-1.$ 
      Then let
      $$\mathcal{N}_{e}^{\prime}=\left\{(\underbrace{0, \ldots, 0}_{k}, x)\in\mathbb{C}^N: x \in \mathcal{N}_{e}\right\},$$
      $$\mathcal{N} = \mathcal{N}_{e}^{\prime} \oplus \operatorname{span}(d),$$
which will be the null space of constructed matrix,    
      where$$d = (\frac{N-4 k}{2^{2}\omega_1}, \frac{N-4 k}{2^{3}\omega_2}, \ldots, \frac{N-4 k}{2^{k+1}\omega_k},  \underbrace{-1, \ldots,-1}_{N-k})\in\mathbb{C}^N.$$
      It is obvious that $\dim(\mathcal{N})=N-m$, then $\dim(\mathcal{N}^{\perp})=m$.

      Assume that
      $$\varphi_{1} = \frac{1}{\rho}(\underbrace{\alpha, \ldots, \alpha}_{k}, 1, \ldots, 1)\in\mathbb{C}^N,$$
      where $\alpha$ and $\rho$ are undetermined coefficients.

      Normalize with $\rho$ so that $\left\|\varphi_{1}\right\|_{2} = 1$, which means $\rho=\sqrt{N+(\alpha^2-1)k}$.
      Let $\varphi_{1} \perp d$, then $\alpha = \frac{(N-k)}{(N-4 k)\sum_{i=1}^k(\frac{1}{2^{i+1}\omega_i})}$, so
      \begin{equation}
      \label{alpha}
          \frac{2(N-k)}{(N-4k)(1-2^{-k})}\leq\alpha\leq\frac{2(N-k)\|\omega\|_\infty}{(N-4k)(1-2^{-k})}.
      \end{equation}
      Since $\varphi_{1} \perp d$ and $x\perp (1, \ldots, 1)$, then $\varphi_{1} \perp \mathcal{N}$. 
      So there exists an orthonormal basis of $\mathcal{N}^{\perp}$ with $\varphi_{1}$ being one of the basis vectors. Let $\Phi\in\mathbb{C}^{m\times N}$ be the matrix whose rows are this orthonormal basis, where $\varphi_1$ is the first row.

      Thus the null space of $\Phi$ is $\mathcal{N}$, and since the rows of $\Phi$ are orthonormal, $\Phi \Phi^{*} = I$ and $\lambda(\Phi) = 1$.
        
      (2) prove the matrix satisfies $\omega$-NSP

      Since $\mathcal{N}=\mathcal{N}_{e}^{'}\bigoplus d$, to prove $\Phi$ satisfies $\omega$-NSP, we need to prove that for any subset $T\subset[N]$ with $\sum_{i\in T}\omega_i^2\leq s$ and vector $h\in \mathcal{N}_{e}^{'}$, $b=h+d$ meets $\left\|b_{T}\right\|_{\omega,1}<\left\|b_{T^c}\right\|_{\omega,1}$.
      Let set $I=\{1,\ldots,k\}$.
      $\|d_{I}\|_{\omega,1}=\frac{N-4k}{2}(1-\frac{1}{2^{k}})<\frac{(N-k)}{2}
      =\frac{1}{2}\sum_{i\in I^c}(-d_i)=\frac{1}{2}\sum_{i\in I^c}(-d_i-h_i)$.
    The last equality is because $h\perp e$, and $h_i=0$, $i=1,\ldots,k$, then $\sum_{i\in I^c}h_i=0$.

     This together with the observation $\omega_{i}\geq 1$, further implies
      \begin{equation}
      \label{1}
           \|b_{I}\|_{\omega,1}=\|d_{I}\|_{\omega,1}\leq \frac{1}{2}\sum_{i=k+1}^{N}|\omega_{i}||d_{i}+h_{i}|
      =\frac{1}{2}\|d_{I^{c}}+h_{I^{c}}\|_{\omega,1}=\|b_{I^c}\|_{\omega,1}. 
      \end{equation}

      For any subset $S\subset I^{c}$ with $\sum_{i\in S}\omega_i^2\leq s$, since $b_{I^{c}}=d_{I^{c}}+h_{I^{c}}\in\ker(A)$ and $A$ satisfies $\omega$-NSP with $\gamma=1/3$, we have
      \begin{equation}
      \label{2}
 \|b_{S}\|_{\omega,1}=\|(b_{I^{c}})_{S}\|_{\omega,1}
      \le\frac{1}{3}\|(b_{I^{c}})_{S^{c}}\|_{\omega,1}
      =\frac{1}{3}\|b_{I^{c}\cap S^{c}}\|_{\omega,1}.          
      \end{equation}

      For any subset $T\subset[N]$ with $\sum_{i\in T}\omega_i^2\leq s$, by \eqref{1} and \eqref{2}, we conclude
      \begin{align*}
        \left\|b_{T}\right\|_{\omega,1} &=\left\|b_{T \cap I}\right\|_{\omega,1}+\left\|b_{T \cap I^{c}}\right\|_{\omega,1}\\
        &\leq\left\|b_{I}\right\|_{\omega,1}+\left\|b_{T \cap I^{c}}\right\|_{\omega,1} \\
        &<\frac{1}{2}\left\|b_{I^{c}}\right\|_{\omega,1}+\left\|b_{T \cap I^{c}}\right\|_{\omega,1} \\
        &=\frac{1}{2}\left\|b_{I^{c} \cap T^{c}}\right\|_{\omega,1}+\frac{1}{2}\left\|b_{I^{c} \cap T}\right\|_{\omega,1}+\left\|b_{T \cap I^{c}}\right\|_{\omega,1}\\
        &=\frac{1}{2}\left\|b_{I^{c} \cap T^{c}}\right\|_{\omega,1}+\frac{3}{2}\left\|b_{I^{c} \cap T}\right\|_{\omega,1} \\
        & \leq \frac{1}{2}\left\|b_{I^{c} \cap T^{c}}\right\|_{\omega,1}
        +\frac{1}{2}\left\|b_{I^{c} \cap T^{c}}\right\|_{\omega,1} \leq \left\|b_{I^{c} \cap T^{c}}\right\|_{\omega,1} \leq \left\|b_{T^{c}}\right\|_{\omega,1} .
      \end{align*}
      
    $\Phi$ satisfies $\omega$-NSP of order $s$ has been proven.

      (3) prove the matrix does not satisfy $\omega$-RIP-NSP
      Let $x_{0}=\left(\frac{N-4k}{2^{2}\omega_1},\frac{N-4 k}{2^{3}\omega_2}, \ldots, \frac{N-4 k}{2^{k+1}\omega_k}, 0, \ldots, 0\right)^{T}\in\mathbb{C}^N$, $z=(\rho,0,\ldots,0)^{T}\in\mathbb{C}^m$ and $y=\Phi x_{0}-z$.
      Consider the weighted $\ell_1$ minimization problem
      \begin{equation}
      \label{l1}
    \min_{x\in\mathbb{C}^N}\|x\|_{\omega,1}\,\text{subject to}\,\|y-\Phi x\|_{2}\leq\rho.
\end{equation}

      Suppose that 
            $$
      \hat{x}:=x_{0}-\rho \varphi_{1}-d=(\underbrace{-\alpha,-\alpha, \ldots,-\alpha}_{k}, 0, \ldots 0)^{T}.
      $$

      The strategy of the proof is that, assuming that $\Phi$ satisfies $\omega$-RIP-NSP of order $s$, if $\|\Phi \hat{x}-\Phi x_0\|_{2}\leq\rho$ and $\|\hat{x}\|_{\omega,1}\leq\|x_{0}\|_{\omega,1}$ hold, then according to \ref{probd}, there is a stable and robust upper error bound for \eqref{l1}. Additionally, by considering the definitions of $x_0$ and $\hat{x}$, we can directly compute a lower bound on $\|\hat{x}-x_{0}\|_{2}$. However, under certain conditions, there is a contradiction between those two error bounds, indicating that the matrix $\Phi$ does not satisfy $\omega$-RIP-NSP.

    Since $d\in\ker(\Phi)$ and $\Phi\varphi_1=(1,0,\ldots,0)$, we have
                $$
      \|\Phi \hat{x}-\Phi x_0\|_{2}=\left\|\Phi x_{0}-\Phi \rho \varphi_{1}-\Phi d-\Phi x_{0}\right\|_{2}=\left\|-\rho \Phi \varphi_{1}\right\|_{2}=\rho.
      $$

      We wish to get an error bound from Proposition \ref{probd}, so $\|\hat{x}\|_{\omega,1}\leq\|x_{0}\|_{\omega,1}$ needs to be proven.\\
      Since $k>1$ and $\|\omega\|_\infty^2\geq1$, we have 
      \begin{align*}
        \|\hat{x}\|_{\omega,1}&=\alpha\sum_{i=1}^k\omega_i=\frac{(N-k)\sum_{i=1}^k\omega_i}{(N-4 k)\sum_{i=1}^k(\frac{1}{2^{i+1}\omega_i})}\leq\frac{2\|\omega\|_\infty^2(N-k)k}{(N-4k)(1-2^{-k})}
        <\frac{4\|\omega\|_\infty^2(N-k)k}{(N-4k)};\\
        \|x_{0}\|_{\omega,1}&=\frac{1}{2}(N-4 k)(1-2^{-k})>\frac{1}{4}(N-4 k).
      \end{align*}
    If $N\geq24\|\omega\|_\infty^2s$, while $s\geq k$, $N\geq24\|\omega\|_\infty^2k$. We conclude that
 $$
    \frac{4\|\omega\|_\infty^2(N-k)k}{(N-4k)}\leq5\|\omega\|_\infty^2k\leq\frac{1}{4}(N-4 k).
 $$

   So if $N\geq24\|\omega\|_\infty^2s$, $\|\hat{x}\|_{\omega,1}\leq\|x_0\|_{\omega,1}$.

       By Proposition \ref{probd}, $\lambda(\Phi)=1$ and $\rho=\sqrt{\alpha^2k+N-k}$, the estimation of the error bound is
      \begin{equation}
        \|\hat{x}-x_{0}\|_{2}^2\leq \left(D_{2}\frac{\sqrt{1+\delta_{\omega,3s}}}{\lambda(\Phi)}\sqrt{N_\nu}\rho\right)^2=C'N_\nu(N+(\alpha^2-1)k)
        \leq2C'NN_\nu,
      \end{equation}
      where $C'=D_2^2(1+\delta_{\omega,3s})$ and $D_2=\frac{6\sqrt{1+\delta_{\omega,3s}}}{1-\delta_{\omega,3s}}$ is defined in Example \ref{ex1} and $N_\nu$ as in Lemma \ref{op} represents the number of subsets in a partition of $[N]$ with sparse function $\nu_\omega(S)=\sum_{i\in S}\omega_i^2$.

     By \eqref{nv} and $s\geq k$, we have that $N_\nu\leq\frac{N\|\omega\|_\infty^2}{k}+1$. So
     \begin{equation}
         \label{ebd1}
         \|\hat{x}-x_{0}\|_{2}^2\leq2C'N\left(\frac{N^\|\omega\|_\infty^2}{k}+1\right).
     \end{equation}

      Since $N\geq24\|\omega\|_\infty^2s$ and \eqref{alpha}, implies $2<\alpha<3\|\omega\|_\infty$. The other form of the error bound is 
      \begin{align}
        \left\|\hat{x}-x_{0}\right\|_{2}^{2} &=\left\|\rho \varphi_{1}+ d\right\|_{2}^{2}=\sum_{i=1}^{k}\left(\alpha+\frac{N-4 k}{2^{i+1}\omega_i} \right)^{2} \\
        \label{ebd2}
        &=\sum_{i=1}^{k}\left(\alpha^{2}+2 \alpha \frac{N-4 k}{2^{i+1}\omega_i}+\frac{(N-4 k)^{2}}{4^{i+1}\omega_i^2}\right) \geq \frac{(N-4 k)^{2}}{16\|\omega\|_\infty^2}.
      \end{align}

     If $s>23040\|\omega\|_\infty^{6}$, while $(k+1)\|\omega\|_\infty^2>s$ and $\delta_{\omega,3s}<1/3$, we have $k>160C'\|\omega\|_\infty^4$. Known that $N\geq24\|\omega\|_\infty^2s$, the estimation of error bounds is
\begin{align*}
    \frac{(N-4 k)^{2}}{16\|\omega\|_\infty^2}
    >\frac{N^2}{40\|\omega\|_\infty^2}
    >4C'N^2\frac{\|\omega\|_\infty^2}{k}
    >2C'N\left(\frac{N^\|\omega\|_\infty^2}{k}+1\right).
\end{align*}

     So there is a contradiction between \eqref{ebd1} and \eqref{ebd2},  so $\Phi$ does not satisfy $\omega$-RIP-NSP.
    \end{proof}

    Apart from $\omega$-RIP and $\omega$-RIP-NSP, there is another property that guarantees a stable and robust error bound.
For $\nu_{\omega}(S)=\sum_{i\in S}\omega_i^2$ case, there is a version of the weighted robust null space property involved in the proofs, denoted by $\omega$-robust-NSP.
\begin{definition}[$\omega$-robust-NSP]
        \label{robustnsp}
    For given weights $\omega$, constants $\rho<1,\,\gamma>0$, the matrix $A\in \mathbb{C}^{m\times{N}}$ is said to satisfy the weighted robust null space property of order $s$ if it satisfies:
    \begin{equation}
        \|\upsilon_{S}\|_{2}\le\frac{\rho}{\sqrt{s}}\|\upsilon_{S^{C}}\|_{\omega,1}+\gamma\|A\upsilon \|_{2}\;\mbox{for all}\;\upsilon\in\mathbb{C}^{N}\mbox{and all}\;S\subset[N]\;\mbox{with}\;\omega(S)\le s.
    \end{equation}
    \end{definition}

It is currently known that $\omega$-RIP implies  $\omega$-robust-NSP and $\omega$-robust-NSP implies $\omega$-NSP. 
In \cite{rauhut2016}, it was proven that if the matrix satisfies $\omega$-RIP of order$3s$ with $\delta_{\omega,3s}<1/3$, then the matrix also satisfies $\omega$-robust-NSP of order $s$ with constants $\rho=2\delta_{\omega,3s}/(1-\delta_{\omega,3s})<1$ and $\gamma=\sqrt{1+\delta_{\omega,3s}}/(1-\delta_{\omega,3s})$, namely $\rho=\sqrt{\frac{(2+\gamma)(1+\gamma)}{2}}$.

In the following theorem, we can show that there exists a matrix satisfying $\omega$-RIP-NSP but not $\omega$-robust-NSP. However, whether there exists a matrix satisfying $\omega$-robust-NSP but not $\omega$-RIP-NSP is still open.

The set of matrices $\{U\Phi: \text{U is invertible}\}$ share the same null space but can have dramatically different RIP constants.
To prove that there exist  matrices fulfilling $\omega$-RIP-NSP but not $\omega$-robust-NSP, we need to show for matrices satisfying $\omega$-RIP, $U\Phi$ will violate $\omega$-robust-NSP with some invertible matrices $U$.
\begin{proposition}
    \label{r2n}
    For any given $\rho<1$ and $\gamma>0$ such that $\rho=\sqrt{\frac{(2+\gamma)(1+\gamma)}{2}}$, there exists a matrix $\Phi$ such that $\Phi$ satisfies $\omega$-RIP-NSP of order $3s$ with $\delta_{\omega,3s}=\frac{\rho}{2+\rho}<1/3$ but not $\omega$-robust-NSP of order $s$ with $\rho$ and $\gamma$.
\end{proposition}
    \begin{proof}
        For any given $\rho<1$ and $\gamma>0$, suppose that there exists a matrix $\Psi$ satisfying $\omega$-RIP with $\delta_{\omega,3s}=\frac{\rho}{2+\rho}$, then $\Psi$ also satisfies $\omega$-robust-NSP with $\rho=2\delta_{\omega,3s}/(1-\delta_{\omega,3s})$ and $\gamma=\sqrt{1+\delta_{\omega,3s}}/(1-\delta_{\omega,3s})$, i.e.  
        $$
        \|\upsilon_{S}\|_{2}\le\frac{\rho}{\sqrt{s}}\|\upsilon_{S^{C}}\|_{\omega,1}+\gamma\|\Psi\upsilon \|_{2}\;\mbox{for all}\;\upsilon\in\mathbb{C}^{N}\mbox{and all}\;S\subset[N]\;\mbox{with}\;\omega(S)\le s.
    $$ 
   Since matrices in $\{U\Phi:\;\text{U is invertible}\}$ share the same null space, for any invertible matrix $U$, $U\Psi$ satisfies $\omega$-RIP-NSP of order $3s$ with $\delta_{\omega,3s}$.

    Then for any vector $x\in\mathbb{C}^{N}/\ker(\Psi)$, if $\lambda(U)<\frac{\|x_S\|_2-\frac{\rho}{\sqrt{s}}\|x_{S^{C}}\|_{\omega,1}}{\gamma\|\Psi x\|_{2}}$, where $\lambda(U)$ is the smallest positive singular value of $U$, denote $U\Psi$ as $\Phi$.
    \begin{align*}
        \frac{\rho}{\sqrt{s}}\|x_{S^{C}}\|_{\omega,1}+\gamma\|\Phi x \|_{2}
    \leq\frac{\rho}{\sqrt{s}}\|x_{S^{C}}\|_{\omega,1}+\gamma\lambda(U)\|U\Phi x\|_{2}
    <\|x_S\|_2,
    \end{align*} 
    So $\Phi$ violates $\omega$-robust-NSP of order $s$ with $\rho$ and $\lambda$.
    \end{proof}

On the other hand, constructing explicit examples of matrices that satisfy robust NSP but not RIP is a challenging problem even in the unweighted case. One reason for this difficulty is that most of the known constructions of such matrices rely on random matrix theory and probabilistic arguments, which makes it hard to explicitly verify whether a given matrix satisfies the property or not. Despite these challenges, the search for explicit examples of matrices that satisfy robust NSP but not RIP remains an active area of research in compressed sensing and related fields. 
\subsection{Case 2: $\nu_{\omega}(S)=|S|$ with $\|\omega\|_\infty\leq1$}
In this subsection, for another type of weights mentioned in above Subsection \ref{ex2}, we also desire to identify a special class of matrices that satisfy $\omega$-NSP but not $\omega$-RIP-NSP under certain conditions.

In this case $\nu_\omega(S)=|S|$, thus $\omega$-RIP is equivalent to the standard RIP condition. However, if the matrix $\Phi$ satisfies $\omega$-NSP, for any vector $v\in\ker(\Phi)$ and $S\subset[N]$ with $|S|\leq s$, $\|v_S\|_{\omega,1}<\|v_{S^c}\|_{\omega,1}$. But for NSP, it required $\ell_1$ norm on $S$ is smaller than on $S^c$. So the required constant for RIP to imply $\omega$-NSP is distinct from the constant that enables RIP to imply NSP. The result obtained in this subsection remains meaningful.

In Theorem \ref{01th} proved that the $\omega$-RIP of order $2s$ with $\delta_{2s}<\frac{\gamma}{\gamma+2}$ can guarantee the exact recovery in the noiseless case. And in the proof of this result, $\omega$-RIP implies $\omega$-NSP has been proven. In the unweighted case, the matrix satisfies $\omega$-RIP of order $2s$ with $\delta<1/3$ can imply exact recovery. And $\frac{\gamma}{\gamma+2}<1/3$ while $\gamma<1$, so the bound in this weighted case is smaller than that in the unweighted case.

 In \cite{cahill2016}, Cahill et al. provided a quite general class of DFT matrices that satisfy the $\omega$-RIP.
\begin{theorem}
\label{dftrip}
     For any fixed absolute constant  $\delta<1$ , when  $m,\;N,\;s$  satisfy  $m \geq C_{1}  ts  \log ^{4} N$ , the random partial Fourier matrix formed by randomly choosing $m$ rows from an  $N \times N$  DFT matrix satisfies the $\omega$-RIP with constant  $\delta_{2 s} \leq \delta$  with probability  $1-5 e^{-\frac{\delta^{2} t}{C_{2}}}$ , where $ C_{1}, C_{2}$  are absolute constants.
\end{theorem}

Combining the two theorems mentioned above, we obtain the following corollary.
\begin{corollary}
\label{c2}
    For any given positive constant $\gamma<1$, if constants $C_1$, $C_2$, $m,\,s,\,N$ satisfy $N/2> m\geq 2 \log 10 C_{1} C_{2}\frac{(\gamma+2)^2}{\gamma^2}s\log^{4}(N)$, there exists $F\in \mathbb{C}^{m\times{N}}$ with $\omega$-NSP of order $s$ with $\gamma$ and $(1,1,\ldots,1)^{T}\in \ker(F)$.
\end{corollary}
    
\begin{proof}
      Let $\delta=\frac{\gamma}{\gamma+2}$. By Theorem \ref{dftrip}, the partial DFT matrix  $F \in \mathbb{C}^{m \times N}$  satisfies RIP with $ \delta_{2 s}< \delta$  with probability at least  $p_{1}=1-5 e^{\frac{-\delta^{2} t}{C_{2}}}$  if  $m \geq 2 C_{1} t s \log ^{4} N$, where $t$ is a constant to be chosen later. According to \ref{01th}, if $F$ satisfies RIP with $\delta_{2s}<\frac{\gamma}{\gamma+2}$, then it satisfies $\omega$-NSP of order $s$.

      We further require  $(1, \ldots, 1)^{T} \in \operatorname{ker}(F)$, which is equivalent to not selecting the first row of the DFT matrix. The probability of not selecting the first row  $p_{2}=(N-m) / N$ .

      The existence of  $F$ such that it has RIP with  $\delta_{2 s} \leq \delta$  and  $(1, \ldots, 1)^{T} \in \operatorname{ker}(F)$  is guaranteed if  $p_{1}>1-p_2$. If $m \geq 2 \log(10)C_{1} C_{2}\frac{(\gamma+2)^2}{\gamma^2}s$, we have $t\geq\log(10)C_2\frac{(\gamma+2)^2}{\gamma^2}$, so $p_1\geq\frac{1}{2}$. Since $m< N/2$, $p_2>\frac{1}{2}$. While $p_1\geq\frac{1}{2}$ and $p_2>\frac{1}{2}$, $p_{1}>1-p_2$ holds.
\end{proof}

Similar to Theorem \ref{1.3}, by considering the parameters specified in Corollary \ref{c2}, the desired matrices can be constructed with suitable adjustments.
\begin{theorem}
      \label{result2}
      For any given $\gamma\in(\frac{3}{4},1)$ and the weight $\omega=[\omega_1,\ldots,\omega_N]^T\in [\gamma,1]^N$, if $m,\;s,\;N\in\mathbb{N}$ satisfy that $N/2\geq m\geq242\log 10 C_{1} C_{2}s\log^{4}(N)$, $N\geq24s$ and $s\geq3717120$, where $C_1$ and $C_2$ are constants as in Theorem \ref{dftrip}, then there exists a sensing matrix $\Phi\in\mathbb{C}^{m\times{N}}$that satisfies $\omega$-NSP of order $s$, but does not satisfy $\omega$-RIP-NSP of order $2s$ with $\delta_{2s}<\frac{1}{11}$.
\end{theorem}

 \begin{proof}
 The proof follows a similar approach as that of Theorem \ref{1.3}. From Corollary \ref{c2}, the matrix with $\omega$-NSP is obtained. We construct a matrix depending on the null space of the former matrix and then prove it satisfies $\omega$-NSP of order $s$ but not $\omega$-RIP-NSP of order $2s$ with $\delta_{2s}$.

      (1) construct the matrix

      Let $\gamma\in(\frac{3}{4},1)$ and the weight $\omega=[\omega_{1},\ldots,\omega_N]^T\in\mathbb{C}^N$, where $\omega_i\in[\gamma,1],\;i=1,\ldots,N.$
      By the assumption that $m\geq242\log10 C_{1} C_{2}s\log^{4}(N)$, while $1\leq121\log10 C_{1} C_{2}[2\log^4(N)-\log^4(N-s)]$, we have $(m-s)\geq121\log 10 C_{1} C_{2}s\log^{4}(N-s)$.
      By Corollary \ref{c2}, for weight $\tilde{\omega}=[\omega_{s+1},\ldots,\omega_N]^T$, there exists a matrix $F\in \mathbb{C}^{(m-s)\times (N-s)}$ satisfies $\omega$-NSP of order $s$ with $\frac{1}{5}$ and $e=(1,1,\ldots,1)^{T}\in \ker(F)$.
    
      Define $$\mathcal{N}_{e}=\{x\in \ker (F): x\bot e\}.$$
      Since $F$ is sampled by a unitary matrix, $rank(F)=m-s$, so the dimension of $\ker(F)$ is $N-m$, $\dim(\mathcal{N}_{e})=N-m-1.$ 
      Let
      $$\mathcal{N}_{e}^{\prime}=\left\{(\underbrace{0, \ldots, 0}_{s}, x)\in\mathbb{C}^N: x \in \mathcal{N}_{e}\right\},$$
      $$\mathcal{N} = \mathcal{N}_{e}^{\prime} \oplus \operatorname{span}(d),$$
      where$$d = \left(\frac{N-4 s}{2^{2}\omega_1}, \frac{N-4 s}{2^{3}\omega_2}, \ldots, \frac{N-4 s}{2^{s+1}\omega_s},  \underbrace{-1, \ldots,-1}_{N-s}\right)\in\mathbb{C}^N.$$
      $\mathcal{N}$ will be the null space of the constructed matrix.
      It is obvious that $\dim(\mathcal{N})=N-m$, then $\dim(\mathcal{N}^{\perp})=m$.

      Assume that
      $$\varphi_{1} = \frac{1}{\rho}(\underbrace{\alpha, \ldots, \alpha}_{s}, 1, \ldots, 1)\in\mathbb{C}^N,$$
      where $\rho=\sqrt{N+(\alpha^2-1)s}$ and  $\alpha = \frac{(N-s)}{(N-4s)\sum_{i=1}^s(\frac{1}{2^{i+1}\omega_i})}$, thus $\left\|\varphi_{1}\right\|_{2} = 1$ and $\varphi_{1} \perp d$. By these and $x\perp (1, \ldots, 1)$, we further include that $\varphi_{1} \perp \mathcal{N}$. 
      So there exists an orthonormal basis of $\mathcal{N}^{\perp}$ with $\varphi_{1}$ being one of the basis vectors. Let $\Phi\in\mathbb{C}^{m\times N}$ be the matrix whose rows are this orthonormal basis, where $\varphi_1$ is the first row. Then the null space of $\Phi$ is $\mathcal{N}$, and since the rows of $\Phi$ are orthonormal, $\Phi \Phi^{*} = I$ and $\lambda(\Phi) = 1$.
        
      (2) prove the matrix satisfies $\omega$-NSP

      Since $\mathcal{N}=\mathcal{N}_{e}^{'}\bigoplus d$, to prove $\Phi$ satisfies $\omega$-NSP, we only need to prove that for any subset $T\subset[N]$ with $|T|\leq s$ and vector $h\in \mathcal{N}_{e}^{'}$, $b=h+d$ meets $\left\|b_{T}\right\|_{\omega,1}<\left\|b_{T^c}\right\|_{\omega,1}$.
      Let set $I=\{1,\ldots,s\}$.
      $\|d_{I}\|_{\omega,1}=\frac{N-4s}{2}(1-\frac{1}{2^{s}})<\frac{N-s}{2}
      =\frac{1}{2}\sum_{i\in I^c}(-d_i)=\frac{1}{2}\sum_{i\in I^c}(-d_i-h_i)$.
    The last equality is because $h\perp e$, and $h_i=0$, $i=1,\ldots,s$, then $\sum_{i\in I^c}h_i=0$.

   Since $\omega_i\geq\gamma$, we have
      $$\frac{1}{2}\sum_{i\in I^c}(-d_i-h_i)\leq\frac{1}{2}\sum_{i\in I^c}|d_{i}+h_{i}|
      \leq\frac{1}{2\gamma}\|d_{I^{c}}+h_{I^{c}}\|_{\omega,1}.$$

      For any subset $S\subset I^{c}$ with $|S|\leq s$, since $b_{I^{c}}=d_{I^{c}}+h_{I^{c}}\in\ker(F)$ and $F$ satisfies $\omega$-NSP with $\frac{1}{5}$, we have
      \begin{equation}
      \label{2}
 \|b_{S}\|_{\omega,1}=\|(b_{I^{c}})_{S}\|_{\omega,1}
      \le\frac{1}{5}\|(b_{I^{c}})_{S^{c}}\|_{\omega,1}
      =\frac{1}{5}\|b_{I^{c}\cap S^{c}}\|_{\omega,1}.          
      \end{equation}

      For any subset $T\subset[N]$ with $|T|\leq s$ and $\gamma\in(\frac{3}{4},1)$, we conclude
      \begin{align*}
        \left\|b_{T}\right\|_{\omega,1} &=\left\|b_{T \cap I}\right\|_{\omega,1}+\left\|b_{T \cap I^{c}}\right\|_{\omega,1}\\
        &\leq\left\|b_{I}\right\|_{\omega,1}+\left\|b_{T \cap I^{c}}\right\|_{\omega,1} \\
        &<\frac{1}{2\gamma}\left\|b_{I^{c}}\right\|_{\omega,1}+\left\|b_{T \cap I^{c}}\right\|_{\omega,1} \\
        &=\frac{1}{2\gamma}\left\|b_{I^{c} \cap T^{c}}\right\|_{\omega,1}+\frac{1}{2\gamma}\left\|b_{I^{c} \cap T}\right\|_{\omega,1}+\left\|b_{T \cap I^{c}}\right\|_{\omega,1}\\
        &=\frac{1}{2\gamma}\left\|b_{I^{c} \cap T^{c}}\right\|_{\omega,1}+\frac{1+2\gamma}{2\gamma}\left\|b_{I^{c} \cap T}\right\|_{\omega,1} \\
        & \leq \frac{1}{2\gamma}\left\|b_{I^{c} \cap T^{c}}\right\|_{\omega,1}
        +\frac{1+2\gamma}{10\gamma}\left\|b_{I^{c} \cap T^{c}}\right\|_{\omega,1}< \left\|b_{I^{c} \cap T^{c}}\right\|_{\omega,1} \leq \left\|b_{T^{c}}\right\|_{\omega,1} .
      \end{align*}
      
    $\Phi$ satisfies $\omega$-NSP of order $s$ has been proven.

      (3) prove the matrix does not satisfy $\omega$-RIP-NSP
      Let $x_{0}=\left(\frac{N-4s}{2^{2}\omega_1},\frac{N-4s}{2^{3}\omega_2}, \ldots, \frac{N-4s}{2^{s+1}\omega_s}, 0, \ldots, 0\right)^{T}\in\mathbb{C}^N$, $z=(\rho,0,\ldots,0)^{T}\in\mathbb{C}^m$, $y=\Phi x_{0}-z$ and $\|z\|_{2}=\rho$.
      Consider the weighted $\ell_1$ minimization problem
      \begin{equation}
      \label{l1}
    \min_{x\in\mathbb{C}^N}\|x\|_{\omega,1}\,\text{subject to}\,\|y-\Phi x\|_{2}\leq\rho.
\end{equation}

      Suppose that 
            $$
      \hat{x}:=x_{0}-\rho \varphi_{1}-d=(\underbrace{-\alpha,-\alpha, \ldots,-\alpha}_{s}, 0, \ldots 0)^{T}.
      $$

      Assume that $\Phi$ satisfies $\omega$-RIP-NSP of order $s$, if $\|\Phi \hat{x}-\Phi x_0\|_{2}\leq\rho$ and $\|\hat{x}\|_{\omega,1}\leq\|x_{0}\|_{\omega,1}$ hold, then according to \ref{probd}, there is a stable and robust upper error bound for \eqref{l1}. Additionally, by considering the definitions of $x_0$ and $\hat{x}$, we can directly compute a lower bound on $\|\hat{x}-x_{0}\|_{2}$. However, under certain conditions, there is a contradiction between those two error bounds, indicating that the matrix $\Phi$ does not satisfy $\omega$-RIP-NSP.

    Since $d\in\ker(\Phi)$ and $\Phi\varphi_1=(1,0,\ldots,0)$, we have
                $$
      \|\Phi \hat{x}-\Phi x_0\|_{2}=\left\|\Phi x_{0}-\Phi \rho \varphi_{1}-\Phi d-\Phi x_{0}\right\|_{2}=\left\|-\rho \Phi \varphi_{1}\right\|_{2}=\rho.
      $$

      We wish to get an error bound from Proposition \ref{probd}, so $\|\hat{x}\|_{\omega,1}\leq\|x_{0}\|_{\omega,1}$ needs to be proven.\\
      Since $s>1$ and $\|\omega\|_\infty^2\in[\gamma,1]$, known that 
      \begin{align*}
        \|\hat{x}\|_{\omega,1}&=\alpha\sum_{i=1}^k\omega_i\leq\alpha s
        =\frac{(N-s)s}{(N-4s)\sum_{i=1}^s(\frac{1}{2^{i+1}\omega_i})}
        \leq\frac{2s(N-s)}{(N-4 s)(1-2^{-s})}<\frac{4(N-s)s}{(N-4s)};\\
        \|x_{0}\|_{\omega,1}&=(N-4s)\sum_{i=1}^s(\frac{1}{2^{i+1}\omega_i})
        \geq\frac{1}{2}(N-4s)(1-2^{-s})>\frac{1}{4}(N-4s).
      \end{align*}
    If $N\geq24s$, 
    \begin{equation}
        \|\hat{x}\|_{\omega,1}<\frac{4(N-s)s}{(N-4s)}\leq\frac{23}{5}s
        <5s\leq\frac{1}{4}(N-4s)<\|x_{0}\|_{\omega,1}.
    \end{equation}

   So if $N\geq24s$, $\|\hat{x}\|_{\omega,1}\leq\|x_0\|_{\omega,1}$.

      By Proposition \ref{probd} and Theorem \ref{01th}, $\lambda(\Phi)=1$ and $\rho=\sqrt{\alpha^2s+N-s}$, the estimation of the error bound is
      \begin{equation}
    \label{ebd11}          
     \left \|\hat{x}-x_{0} \right \| _{2}\le B_2\frac{\sqrt{1+\delta_{2s}}\sqrt{\frac{N}{s}+1}\rho}{\lambda(\Phi)}
     =B_2\sqrt{1+\delta_{2s}}\sqrt{\frac{N}{s}+1}\sqrt{N+(\alpha^{2}-1)s},
    \end{equation}
     
      where $\delta_{2s}<\frac{1}{11}$ and $B_2$ is defined in \eqref{b2}.

    Since $N\geq24s$ and $\alpha = \frac{(N-s)}{(N-4s)\sum_{i=1}^s(\frac{1}{2^{i+1}\omega_i})}$, implies $2\gamma<\alpha<5$. 
     Then square \eqref{ebd11} implies
      \begin{equation}
      \label{ebd31}
        \|\hat{x}-x_{0}\|_{2}^{2}\leq C'\left(\frac{N}{s}+1\right)(N+(\alpha^{2}-1)s)\leq2C'N\left(\frac{N}{s}+1\right).
      \end{equation}
      where $C'=B_2^2(1+\delta_{2s})$.\\

      We can also calculate the error bound directly,
      \begin{align}
        \left\|\hat{x}-x_{0}\right\|_{2}^{2} &=\left\|\rho \varphi_{1}+ d\right\|_{2}^{2}=\sum_{i=1}^{s}\left(\alpha+\frac{N-4s}{2^{i+1}\omega_i} \right)^{2} \\
        \label{ebd21}
        &=\sum_{i=1}^{s}\left(\alpha^{2}+2 \alpha \frac{N-4s}{2^{i+1}\omega_i}+\frac{(N-4 s)^{2}}{4^{i+1}\omega_i^2}\right) \geq \frac{(N-4 s)^{2}}{16}.
      \end{align}

     If $s>3717120$, while $\delta_{2s}<\frac{1}{11}$, we have $s>160C'$. Known that $N\geq24s$, the estimation of error bounds is
\begin{align*}
    \frac{(N-4s)^{2}}{16}\geq\frac{N^2}{40}>2C'N(\frac{N}{s}+1).
\end{align*}

     So there is a contradiction between \eqref{ebd31} and \eqref{ebd21},  so $\Phi$ does not satisfy $\omega$-RIP-NSP.
    \end{proof}

\newpage
\nocite{*}
\bibliographystyle{IEEEtran}
\bibliography{template}

\end{document}